\documentclass[12pt]{article} 

\usepackage[margin=3cm]{geometry}
\usepackage[english]{babel}
\usepackage{hyperref}
\usepackage{amsmath}
\usepackage{amsthm}
\usepackage{amssymb}
\usepackage{ucs}
\usepackage{enumerate}
\usepackage[utf8x]{inputenc}
\usepackage[T1]{fontenc}
\usepackage{color}
\usepackage{tikz}
\usepackage{diagbox}

\DeclareMathOperator{\rng}{\mathrm{rng}}
\DeclareMathOperator{\supp}{\mathrm{supp}}

\newcommand{\MAlg}{\mathrm{MAlg}}
\newcommand{\Aut}{\mathrm{Aut}}

  \newcommand{\R}{\mathbb R}
  \newcommand{\N}{\mathbb N}
  
  \newcommand{\Z}{\mathbb Z}
  
  \newcommand{\LL}{\mathrm L}
  \newcommand{\U}{\mathcal U}

 \newcommand{\dom}{\mathrm{dom}\;}

  \newcommand{\inv}{^{-1}}

  \renewcommand{\ker}{\mathrm{Ker}\,}
  
    \newcommand{\spec}{\mathrm{spec}}

  \renewcommand{\leq}{\leqslant}
  \renewcommand{\geq}{\geqslant}
  \newcommand{\abs}[1]{\left\lvert #1\right\rvert}
  \newcommand{\norm}[1]{\left\lVert #1\right\rVert}

  \newcommand{\impl}{\Rightarrow}

  \newcommand{\la}{\left\langle}
  \newcommand{\ra}{\right\rangle}

\newtheorem{thm}{Theorem}[section]
\newtheorem{cor}[thm]{Corollary}
\newtheorem{lem}[thm]{Lemma}
\newtheorem{prop}[thm]{Proposition}
\newtheorem*{claim}{Claim}

\theoremstyle{definition}

\newtheorem{qu}[thm]{Question}

\newtheorem{df}[thm]{Definition}
\newtheorem*{rmq}{Remark}

\newtheorem{exemple}[thm]{Example}

\title{On a measurable analogue of small topological full groups}

\author{François Le Maître\footnote{Research supported by Projet ANR-14-CE25-0004 GAMME.}}

\begin{document}

\maketitle

\begin{abstract}
We initiate the study of a measurable analogue of small topological full groups that we call $\LL^1$ full groups. These groups are endowed with a Polish group topology which  admits a natural complete right invariant metric. We mostly focus on $\LL^1$ full groups of measure-preserving $\Z$-actions which are actually a complete invariant of flip conjugacy. 

We prove that for ergodic actions the closure of the derived group is topologically simple although it can fail to be simple.  We also show that the closure of the derived group is connected, and that for measure-preserving free actions of non-amenable groups the closure of the derived group and the $\LL^1$ full group itself are never amenable. 

In the case of a measure-preserving ergodic $\Z$-action, the closure of the derived group is shown to be the kernel of the index map. If such an action is moreover by homeomorphism on the Cantor space, we show that the topological full group is dense in the $\LL^1$ full group. Using Juschenko-Monod and Matui's results on topological full groups, we conclude  that $\LL^1$ full groups of ergodic $\Z$-actions are amenable as topological groups, and that they are topologically finitely generated if and only if the $\Z$-action has finite entropy. 
\end{abstract}

\tableofcontents

\section{Introduction}

\subsection{\texorpdfstring{$\LL^1$}{L1}  full groups of measure-preserving transformations}

The study of (invertible) measure preserving transformations on standard probability spaces has a long and rich history stemming from its connections with various fields such as statistical mechanics, information theory, Riemannian geometry or number theory. Over the years, various invariants have been introduced in order to classify such transformations, although a result of Foreman, Rudolph and Weiss shows that such a classification cannot ‘‘reasonably'' be obtained for \textit{all} measure-preserving transformations \cite{MR2800720}. 

The most well-behaved measure-preserving transformations are probably the \textbf{compact} ones, i.e. those which arise as translation by a topological generator in a compact abelian group (for instance, an irrational rotation is a compact transformation). 
Indeed, ergodic compact transformations are completely classified up to conjugacy by the countable subgroup of the circle $\mathbb S^1$ consisting of all the eigenvalues of the associated Koopman unitary (for an irrational rotation this is just the group generated by its angle). To be more precise, a theorem of Halmos and von Neumann states that two such transformations $T$ and $T'$ of a standard probability space $(X,\mu)$ are conjugate iff $\spec(T)=\spec(T')$ where $\spec(T)$ is the group consisting of all the $\lambda\in\mathbb S^1$ such that there is a non-null $f\in\LL^2(X,\mu)$ satisfying $f\circ T=\lambda f$. So an ergodic compact transformation $T$ is entirely described by the countable group $\spec(T)$. 

The purpose of this paper is to initiate the study of a \textit{Polish} group $[T]_1$ called the $\LL^1$ full group which one can associate to a measure-preserving transformation $T$. Its main feature resembles that of the spectrum of compact transformations: two ergodic measure-preserving transformations  $T$ and $T'$ are \textit{flip}-conjugate\footnote{Two measure-preserving transformation $T$ and $T'$ are flip-conjugate if there is a measure-preserving transformation $S$ such that $T'=STS\inv$ or $T'=ST\inv S\inv$.} if and only if $[T]_1$ is abstractly (or topologically) \textit{isomorphic} to $[T']_1$.

The idea for this group comes from full groups, which were defined by Dye as subgroups of the group $\Aut(X,\mu)$ of measure-preserving transformations of a standard probability space $(X,\mu)$ satisfying a strong closure property. 
The most basic example of a full group is the  full group of a measure-preserving invertible aperiodic transformation $T$, which consists in all the measure-preserving transformations $U$ such that for almost all $x\in X$, there exists $n\in\Z$ such that $U(x)=T^n(x)$. This full group is denoted by $[T]$ and every $U\in[T]$ is then completely described by the \textbf{cocycle map} $c_U: X\to \Z$ defined by the equation 
$$U(x)=T^{c_U(x)}(x).$$
Note that the full group $[T]$ only depends on the partition of the space into orbits induced by $T$: it is an invariant of \textbf{orbit equivalence}.
A theorem of Dye states that all measure-preserving ergodic transformations are orbit equivalent \cite{MR0131516}, thus making the associated full group a trivial conjugacy invariant for such transformations. To obtain an interesting invariant, we will place an integrability restriction on the cocycle maps.

\begin{df}The \textbf{$\LL^1$ full group} of an aperiodic measure-preserving invertible transformation $T$ on a standard probability space $(X,\mu)$ is the group of all $U\in[T]$ such that 
$$\int_X\abs{c_U(x)}d\mu(x)<+\infty.$$
\end{df}

The $\LL^1$ full group is endowed with a natural metric given by the $\LL^1$ metric when viewing elements of $[T]_1$ as cocycle maps $X\to \Z$. Such a metric is actually complete, separable and right-invariant (see Proposition  \ref{prop: complete separable metric}) so that $\LL^1$ full groups are cli\footnote{A Polish group is cli if it admits a compatible left- (or equivalently right-) invariant metric} Polish groups.

\begin{thm}[see Thm. \ref{thm:l1fullgroup is complete invariant of flip conjugacy}]\label{thmi:l1fullgroup is complete invariant of flip conjugacy}Let $T$ and $T'$ be two measure-preserving ergodic transformation of a standard probability space $(X,\mu)$. Then the following are equivalent: 
\begin{enumerate}[(1)]
\item $T$ and $T'$ are flip-conjugate;
\item the groups $[T]_1$ and $[T']_1$ are abstractly isomorphic;
\item the groups $[T]_1$ and $[T']_1$ are topologically isomorphic. 
\end{enumerate}
\end{thm}

There are two results behind this theorem: the first is a reconstruction result à la Dye which says that any isomorphisms between $\LL^1$ full groups must be a conjugacy by a measure-preserving transformation. The second is Belinskaya's theorem \cite[Cor. 3.7]{belinskaya1968partitions}, which states that if two ergodic transformations share the same $\LL^1$ full group, then they must be flip-conjugate. 

%
%
%

Let us now briefly describe the natural and well-studied analogues of full groups and $\LL^1$ full groups in the context of Cantor dynamics. Given a homeomorphism $T$ of the Cantor space $\mathcal C$, its (topological) full group $[T]_t$ is the group of all homeomorphisms $U$ of the Cantor space such that for all $x\in \mathcal C$ there is $n\in\N$ such that $U(x)=T^n(x)$. This full group turns out to completely determine the partition of the Cantor space into orbits up to a homeomorphisms. Such partitions are entirely described up to homeomorphism in terms of the set of invariant measures: by a theorem of Giordano, Putnam and Skau \cite{MR1363826}, two minimal homeomorphisms $T$ and $T'$ of the Cantor space are orbit equivalent if and only if there is a homeomorphism $S$ such that $S(M_T)=M_{T'}$ where $M_T$ denotes the set of $T$-invariant Borel probability measures. As in the measurable case, an element of the topological full group of an aperiodic homeomorphism $T$ is completely described by the associated cocycle map $c_U: \mathcal C\to \Z$ defined by $U(x)=T^{c_U(x)}(x)$.

The analogue of $\LL^1$ full groups is sometimes called the \textbf{small topological full group} and is also defined by putting a restriction on cocycles. Namely, given a homeomorphism $T$ of the Cantor space $\mathcal C$, its small (topological) full group $[T]_c$ is the group of homeomorphisms $U$ of the Cantor space whose cocycle map $c_U$ is actually continuous. A result of Giordano-Putnam-Skau in complete analogy with Theorem \ref{thmi:l1fullgroup is complete invariant of flip conjugacy} states that when $T$ is minimal, the flip conjugacy class of $T$ is completely determined by its small topological full group $[T]_c$ (see \cite{MR1710743}). In the end, $\LL^1$ full groups can be seen as the missing piece in the following picture. \vspace{0.5cm}

\begin{figure}[htpb]
\begin{center}

\begin{tabular}{| l | c | c |} \hline
 & Topological setup & Measurable setup\\ \hline
 & & \\
Orbit equivalence & $\begin{array}{cl}[T]_t= & \{U\in \mathrm{Homeo}(\mathcal C):  \\ & \forall x, U(x)=T^{c_U(x)}\}\end{array}$  & $\begin{array}{cl}[T]= & \{U\in\Aut(X,\mu): \\ & \forall x, U(x)=T^{c_U(x)}\}\end{array}$ \\
& & \\
\hline
 & & \\
Flip conjugacy & $\begin{array}{cl}[T]_c= & \{U\in [T]_t: \\ & c_U\text{ is continuous}\}\end{array}$ & $\begin{array}{cl}[T]_1= & \{U\in [T]: \\ & c_U\text{ is integrable}\}\end{array}$ \\
  & & \\
\hline \end{tabular}
\end{center}
\end{figure}

\subsection{Some topological properties of $\LL^1$ full groups}

Theorem \ref{thmi:l1fullgroup is complete invariant of flip conjugacy} implies that every ergodic theoretic property of a measure-preserving transformation $T$ should be reflected algebraically in the group $[T]_1$. However since the latter is an uncountable group, it is quite hard to find a purely group-theoretic property allowing us to distinguish two $\LL^1$ full groups. But as we already pointed out, we also have at our disposal a Polish group topology, which provides richer tools to study $[T]_1$. Before we present our results, let us mention a fundamental tool for the study of $\LL^1$ full groups and topological full groups.

Given an aperiodic measure-preserving transformation, one defines the \textbf{index map} $I:[T]_1\to \R$ by putting 
$$I(U)=\int_Xc_U(x)d\mu(x).$$
If $T$ is a minimal homeomorphism, we define the index map the same way picking an invariant measure $\mu$ so as to view $T$ as a measure-preserving transformation. The resulting index map can be shown not to depend on the chosen invariant measure \cite{MR1710743}.

We can now give various descriptions of the closure of the derived group of the $\LL^1$ full group of an ergodic measure-preserving transformation $T$, which we call the \textbf{derived $\LL^1$ full group} of $T$ and denote by $\overline{D([T]_1)}$.

\begin{thm}[see Cor. \ref{cor:index take values in Z}]\label{thm:index take values in Z}Let $T\in\Aut(X,\mu)$ be ergodic. Then the derived $\LL^1$ full group of $T$ is equal to the following four groups:
\begin{itemize}
\item the kernel of the index map,
\item the connected component of the identity,
\item the group generated by periodic elements and
\item the group topologically generated by involutions.
\end{itemize}
\end{thm}

Since the index map actually takes values into $\Z$  we conclude from the first item of the above theorem that the (topological) abelianization of the $\LL^1$ full group of an ergodic transformation is always equal to $\Z$.

In the context of topological full groups of minimal homeomorphisms, it is also true that the index map takes values into $\Z$, but its kernel is often much larger than the derived group. Nevertheless, Matui succesfully computed the abelianization of $[T]_c$ in $K$-theoretic terms and gave examples where this abelianisation is not even finitely generated (see \cite{MR2205435}). So the situation for $\LL^1$ full groups is much simpler with respect to the (topological) abelianization. 

An important feature of the derived group of the topological full group of a minimal homeomorphism is that it is a simple group. The following result is a natural analogue.

\begin{thm}\label{thmintro: topological simplicity}
Let $T$ be a measure-preserving transformation. Then $T$ is ergodic if and only if its derived $\LL^1$ full group $\overline{D([T]_1)}$ is topologically simple. 
\end{thm}

This result is actually true in a greater generality for $\LL^1$ full groups of \textit{graphings} as defined by Levitt \cite{MR1366313}, see section \ref{sec:closed normal subgroups}. We also prove that $\overline{D([T]_1)}$ is \textit{never} simple using the fact that there are many Borel sets whose escape time is not integrable (see section \ref{sec:nonsimple}). Let us now seek finer invariants. \\

The \textbf{topological rank} of a separable topological group is the minimal number of elements needed to generate a dense subgroup. The main result of this paper  is that the topological rank of the $\LL^1$ full group of an ergodic measure-preserving transformation $T$  is related to the \textit{entropy} of $T$. 

\begin{thm}\label{thmintro:topological finite generate}Let $T\in\Aut(X,\mu)$ be ergodic. Then the following are equivalent
\begin{enumerate}[(1)]
\item the  $\LL^1$ full group $[T]_1$ is topologically finitely generated;
\item the derived $\LL^1$ full group $\overline{D([T])_1}$ is topologically finitely generated;
\item the transformation $T$ has finite entropy.
\end{enumerate}
\end{thm}

A crucial tool in the previous theorem is Matui's characterization of finite-generatedness  of topological full groups \cite[Thm. 5.4]{MR2205435}. He proved that the derived group of the topological full group of a minimal homeomorphism $T$ is  finitely generated if and only if $T$ is a subshift. He also characterised finite-generatedness of the whole topological full group of $T$ in terms of a strictly more restrictive condition. In our context the situation is again simpler: Theorem \ref{thmintro:topological finite generate} shows that for an ergodic measure-preserving transformation $T$, the $\LL^1$ full group of $T$ is topologically finitely generated if and only if the closure of  its derived group is. 

In order to use Matui's aforementioned results so as to prove Theorem \ref{thmintro:topological finite generate}, we establish the following important connection between $\LL^1$ full groups and topological full groups.

\begin{thm}[see Thm. \ref{thm: density of the whole topo full group}]\label{propintro:dense full group}Let $T$ be a minimal homeomorphism of the Cantor space and let $\mu$ be an invariant measure. Then $D([T]_c)$ is dense in $\overline{D([T]_1)}$ and $[T]_c$ is dense in $[T]_1$. 
\end{thm}

Our proof of Theorem \ref{thmintro:topological finite generate} also relies on a deep theorem of Krieger which characterizes finite entropy transformations as those which can be realised as uniquely ergodic minimal subshifts \cite{MR0393402}. It would be nice to have a purely ergodic-theoretic proof of Theorem \ref{thmintro:topological finite generate}.\\

Theorem \ref{propintro:dense full group} has another interesting consequence. Recall that a topological group is amenable if whenever it acts continuously on a compact space, the action admits an invariant Borel probability measure. Juschenko and Monod have proven that the topological full group of any minimal homeomorphism of the Cantor space is amenable \cite{zbMATH06203677}. Using the Jewett-Krieger theorem we obtained the following result.

\begin{thm}[{See Thm. \ref{thm:amenable for Z}}]\label{thmintro:L1 full group ergodic Z action is amenable}Let $T$ be an ergodic measure-preserving transformation. Then the $\LL^1$ full group of $T$ is amenable. 
\end{thm}

\subsection{$\LL^1$ full groups of actions of finitely generated groups}

A measure-preserving transformation is nothing but a measure-preserving $\Z$-action, so it is natural to ask which of the above definitions and results generalize to measure-preserving  actions of finitely generated groups. 

\begin{df}Let $\Gamma$ be a finitely generated group acting in a measure-preserving manner on a standard probability space $(X,\mu)$. The full group of this action is denoted by $[\Gamma]$ and defined by $$[\Gamma]=\{T\in\Aut(X,\mu): \forall x\in X, T(x)\in\Gamma\cdot x\}.$$
Let $S$ be a finite generating set for $\Gamma$. Then every $\Gamma$-orbit is naturally equipped with a Schreier graph metric $d_S$, and we let the $\LL^1$ full group of the $\Gamma$-action be the group $[\Gamma]_1$ of all $T\in[\Gamma]$ such that the map $x\mapsto d_S(x,T(x))$ is integrable. 
\end{df}

The definition of this group does not depend on the choice of a finite generating set, and the $\LL^1$ full group $[\Gamma]_1$  has a natural Polish group topology. The following result is a counterpart to Theorem \ref{thmintro:L1 full group ergodic Z action is amenable}. 

\begin{thm}[see Thm. \ref{thm:fg non amenable}]\label{thmi:fg non amenable}Let $\Gamma$ be a finitely generated non amenable group acting freely in a measure-preserving manner on $(X,\mu)$. Then neither $[\Gamma]_1$ nor $\overline{D([\Gamma]_1)}$ are amenable.
\end{thm}

In an upcoming paper, we show that the converse of the above result holds, and we generalize Theorem \ref{thmintro:topological finite generate} by showing that given a free ergodic action of a finitely generated group, the action has finite Rokhlin entropy if and only if the derived $\LL^1$ full group has finite topological rank. 

Finally, we should stress out that $\LL^1$ full groups are \textit{not} invariants of $\LL^1$ orbit equivalence a priori but only of $\LL^\infty$ orbit equivalence (see Corollary \ref{cor: L1 OE but not same L1 full group}). I do not know wether they are \textit{complete} invariants of $\LL^\infty$ orbit equivalence: if $[\Gamma]_1=[\Lambda]_1$, does it follow that $\Gamma$ and $\Lambda$ are $\LL^\infty$ orbit equivalent? 

\subsection{Outline}

This paper is organised as follows. Section \ref{sec:prelim} begins with some general facts from ergodic theory. We then recall the definition of graphings, which provide a natural setup for $\LL^1$ full groups. The section ends with a few basic facts from entropy theory. 

In section \ref{sec:L1 full groups of graphings} we give the definition of the $\LL^1$ full group of a graphing and establish some basic properties such as stability under taking induced transformations. We show that the closure of their derived group is topologically generated by involutions and hence connected. We then establish that $\LL^1$ full groups of ergodic graphings satisfy a reconstruction theorem: any isomorphism between them arises as the conjugacy by a measure preserving transformation.
We also use involutions to prove that the only closed normal subgroups of the closure of their derived group are the obvious ones when the graphing is aperiodic (Theorem \ref{thm: closed normal subgroups}). From this we conclude that the graphing is ergodic if and only if the closure of the derived group of its $\LL^1$ full group is topologically simple. Finally, we relate topological full groups to $\LL^1$ full groups in terms of density. 

 Section \ref{sec: L1 full groups of Z actions} contains the main results on $\LL^1$ full groups of $\Z$-actions. We first prove Theorem \ref{thmi:l1fullgroup is complete invariant of flip conjugacy} and then continue by defining the index map and studying its behaviour with respect to induced transformations. We then study involutions more closely by characterizing the Borel sets which arise as supports of involutions (and more generally those which arise as supports of $n$-\textit{cycles}). After that, we show that  the $\LL^1$ full group of $T\in\Aut(X,\mu)$ is generated by transformation induced by $T$ together with periodic transformations, using the work of Belinskaya. We then show that $\overline{D([T]_1)}$ is never simple. Finally, we prove Theorem \ref{thmintro:L1 full group ergodic Z action is amenable} and Theorem \ref{thmintro:topological finite generate}, using the fact that the topological full group is dense in the $\LL^1$ full group. 

In section \ref{sec: further} we first prove Theorem \ref{thmi:fg non amenable}. We then show that $\LL^1$ full groups are not stable under ‘‘$\LL^1$ cutting and pasting''. After that we make a few observations about $\LL^p$ full groups which are defined analogously and we conclude with some more questions.

%

%
%
%
%
%
%

\section{Preliminaries}\label{sec:prelim}

\subsection{Basic ergodic theory}

Let $(X,\mu)$ be a standard probability space, i.e. a standard Borel space equipped with a Borel non-atomic probability measure $\mu$. All such spaces are isomorphic to the interval $[0,1]$ equipped with the Lebesgue measure, in particular they admit a Borel linear order.

We will always work up to measure zero, e.g. for two Borel sets $A$ and $B$ if we write $A\subseteq B$ we mean $\mu(B\setminus A)=0$. For us a \textbf{partition} of a Borel set $A\subseteq X$ will be a countable family of disjoint Borel subsets of $A$ whose union is equal to $A$ up to measure zero. Note that some of these sets may have measure zero or even be empty. 

A Borel bijection $T:X\to X$ is called a \textbf{measure-preserving transformation} of $(X,\mu)$ if for all Borel $A\subseteq X$, one has $\mu(A)=\mu(T\inv (A))$. We denote by $\Aut(X,\mu)$ the group of all measure-preserving transformations, two such transformations being identified if they coincide up to measure zero. An important fact to keep in mind is that every Borel bijection between full measure Borel subsets of $X$ which is measure-preserving can be seen as an element of $\Aut(X,\mu)$: there is a measure-preserving transformation $T'$ of $(X,\mu)$ such that $T(x)=T'(x)$ for almost all $x\in X$.

\begin{df}A measure-preserving transformation $T$ is \textbf{ergodic} if every Borel set which is a union of $T$-orbits has measure $0$ or $1$. It is \textbf{aperiodic} if all its orbits are infinite. 
\end{df}

A fundamental construction for us will be that of the induced transformation: given $T\in\Aut(X,\mu)$ and $A\subseteq X$ of positive measure, the Poincaré recurrence theorem ensures us that for almost all $x\in A$ there are infinitely many $n\in\N^*$ such that $T^n(x)$. We let $n_{A,T}(x)$ denote the smallest such $n$, also called the \textbf{return time} to $A$. When the transformation $T$ we consider is clear from the context, we will denote the return time simply by $n_A$. 

The transformation $T_A\in\Aut(X,\mu)$ \textbf{induced} by $T$ on $A$ by for all $x\in X$,
$$T_A(x)=\left\{\begin{array}{cc}T^{n_{A,T}}(x) & \text{if }x\in A \\x & \text{else.}\end{array}\right.$$
It can easily be checked that $T_A$ is indeed a measure-preserving transformation, and that moreover for all $x\in A$ the $T_A$-orbit of $x$ is the intersection of the $T$-orbit of $x$ with $A$. Also note that $(T_A)\inv=(T\inv)_A$

The following lemma is a direct consequence of the fact that the orbits of the restriction of  $T_A$ to $A$ are precisely the intersections of the $T$-orbits with $A$. 

\begin{lem}\label{lem: properties of induced equiv to that of T}
Let $S\in\Aut(X,\mu)$ be a measure-preserving transformation and $A\subseteq X$ intersects every $S$-orbit, then $S$ is aperiodic if and only if $S_A$ is aperiodic when restricted to $A$, and $S$ is ergodic if and only if $S_A$ is ergodic when restricted to $A$.
\end{lem}

\begin{df}\label{df: periodic}An element $T\in\Aut(X,\mu)$ is \textbf{periodic} if all its orbits are finite.
\end{df}

If $T$ is periodic, the Borel set of minimums of $T$-orbits (for a fixed Borel total order on $X$) intersects every $T$-orbit at exactly one point: it is a Borel \textbf{fundamental domain} for $T$. Such fundamental domains actually exist if and only if $T$ is periodic.

A nice way to build a new transformation from a periodic one is to compose it with a transformation supported on a fundamental domain. 

\begin{lem}\label{lem: induced from periodic}
Let $T$ be a periodic measure-preserving transformation, and let $A$ be a fundamental domain of $T$. Let $U$ be a measure-preserving transformation supported on $A$. Then $(UT)_A=U=(TU)_A$. 
\end{lem}
\begin{proof}
Let $x\in A$, let $n\in\N^*$ be the least positive integer such that $T^n(x)=x$. Since $A$ is a fundamental domain for $T$, for all $i\in\{0,...,n-1\}$ we have $T^i(x)\not\in A$. Since $U(y)=y$ for all $y\not\in A$, we deduce by induction that  $(UT)^i(x)=T^i(x)\not\in A$ for all $i\in\{1,...,n-1\}$. In particular $(UT)^{n-1}(x)=T^{n-1}(x)$ so $(UT)^n(x)=UT^n(x)=U(x)\in A$, which proves that $(UT)_A=U$ as desired.

For the equality $(TU)_A=U$, one can run a similar argument or apply the previously obtained equality to $T\inv$ and $U\inv$ to get
$(U\inv T\inv)_A=U\inv$,
which by taking inverses yields $(TU)_A=U$.
\end{proof}

\begin{prop}\label{prop: building ergodic and aperiodic from periodic}
Let $T$ be a periodic measure-preserving transformation, and let $A$ be a fundamental domain of $T$. Let $U$ be a measure-preserving transformation supported on $A$. Then the following equivalences hold:
\begin{enumerate}[(1)]
\item The restriction of $U$ to $A$ is aperiodic $\Leftrightarrow$ $UT$ is aperiodic  $\Leftrightarrow$ $TU$ is aperiodic.
\item T restriction of $U$ to $A$ is ergodic  $\Leftrightarrow$ $UT$ is ergodic $\Leftrightarrow$ $TU$ is ergodic. 
\end{enumerate}
\end{prop}
\begin{proof}
It follows from the proof of Lemma \ref{lem: induced from periodic} that $A$ intersects every $TU$ orbit and every $UT$ orbit. The conclusion then follows from Lemmas \ref{lem: induced from periodic} and  \ref{lem: properties of induced equiv to that of T}.
\end{proof}

Let us observe that one can go the other way around in the above constructions: starting from an aperiodic transformation $T$ and a Borel set $A$ which intersects almost every $T$-orbit, the map $T_A\inv T$ is periodic with $A$ as a Borel fundamental domain and same return time to $A$ as $T$. This fact is the key idea for the following result of Belinskaya which we won't use but which is  an important step towards to the proof of her Theorem   \ref{thm: belink L1OE implies flip conj}. 

\begin{prop}[{Belinskaya, \cite[Thm. 3.6]{belinskaya1968partitions}}]
Let $S$ and $T$ be two measure-preserving transformations, let $A$ be a Borel subset which intersects every $S$ and every $T$-orbit. Suppose that $S_A=T_A$ and that the return times of $S$ and $T$ to $A$ are the same, i.e. $n_{A,S}=n_{A,T}$. Then $S$ and $T$ are conjugate. 
\end{prop}

We will use the following well-known proposition several times. A proof is provided for completeness.

\begin{prop}\label{prop: small rohlin tower}
Let $T\in\Aut(X,\mu)$. There is a partition $(A_1,A_2,B_1,B_2,B_3)$ of $\supp T$ such that $T(A_1)=A_2$, $T(B_1)=B_2$ and $T(B_2)=B_3$.
\end{prop}
\begin{proof}
Since $X$ is a standard Borel space, there is a countable partition $(X_i)_{i\in\N}$ of $\supp T$ into Borel sets such that for all $i\in\N$, $T(X_i)$ is disjoint from $X_i$ (see for instance \cite[Lem. 5.1]{Eisenmann:2014oq}).

We then define by recurrence an increasing family of sets $(Y_i)_{i\in\N}$ by letting $Y_0=X_0$ and then for all $i\in\N$ $$Y_{i+1}=Y_i\sqcup \{x\in X_{i+1}: T(x)\not\in Y_i\text{ and } T\inv(x)\not\in Y_i\}.$$

Let $Y=\bigcup_{n\in\N} Y_n$, then by construction $Y$ and $T(Y)$ are disjoint. Moreover, we have $\supp T=T\inv(Y)\cup Y\cup T(Y)$. Indeed, if $x\in \supp T \setminus Y$ there is $i\in\N$ such that $x\in X_{i+1}$, and since $x\not\in Y_{i+1}$ either $T(x)\in Y_i$ or $T\inv(x)\in Y_i$ so that in any case $x\in T\inv(Y)\cup T(Y)$. 

Since $\supp T$ is $T$-invariant, we conclude that $\supp T=Y\cup T(Y)\cup T^2(Y)$. 
Now let $A_1=\{x\in Y: T^2(x)\in Y\}$ and $B_1=\{Y\setminus A_1\}$. Then since $Y$ and $T(Y)$ are disjoint, the sets $A_1, A_2:=T(A_1),B_1,B_2:=T(B_1)$ are disjoint. Let $B_3:=T^2(B_1)$. By the definition of $B_1$ the set $B_3$ is disjoint from $Y$. Moreover since $T(B_1)$ is disjoint from $Y$ we also have $T^2(B_1)$ disjoint from $T(Y)$ so that $B_3$ is actually disjoint from $Y\sqcup T(Y)$. We conclude that $A_1, A_2, B_1, B_2, B_3$ are as desired.
\end{proof}

\begin{rmq}Note that the proof does not use that $T$ is measure-preserving and actually works in the purely Borel context.
\end{rmq}

\subsection{Graphings, equivalence relations and full groups}\label{sec:graphings}

If $(X,\mu)$ is a standard probability space, and $A,B$ are Borel subsets of $X$, a \textbf{partial isomorphism} of $(X,\mu)$ of \textbf{domain} $A$ and \textbf{range} $B$ is a Borel bijection $f: A\to B$ which is measure-preserving for the measures induced by $\mu$ on $A$ and $B$ respectively. We denote by $\dom f=A$ its domain, and by $\rng f=B$ its range. Given two partial isomorphisms $\varphi_1: A\to B$ and $\varphi_2: C\to D$, we define their composition $\varphi_2\circ \varphi_1$ as the map $\varphi_1\inv(B\cap C)\to \varphi_2(B\cap C)$ given by $\varphi_2\circ\varphi_1(x)=\varphi_1\varphi_2(x)$. We also define the \textbf{inverse} $\varphi\inv$ of a partial isomorphism $\varphi: A\to B$ by $\varphi\inv: B\to A$ and the equation $\varphi\inv\circ \varphi=\mathrm{id}_A$.

By definition a \textbf{graphing} is a countable set of partial isomorphisms. Every graphing $\Phi$ \textbf{generates} a \textbf{measure-preserving equivalence relation} $\mathcal R_\Phi$, defined to be the smallest equivalence relation containing $(x,\varphi(x))$ for every $\varphi\in \Phi$ and $x\in\dom\varphi$. Given a graphing $\Phi$, a $\Phi$\textbf{-word} is a composition of finitely many elements of $\Phi$ or their inverses. Obviously $(x,y)\in\mathcal R_\Phi$ iff and only if there exists a $\Phi$-word $w$ such that $y=w(x)$. Given a Borel set $A$, its $\Phi$\textbf{-closure} is the union of all $w(A)$ where $w$ is a $\Phi$-word. Say that $\Phi$ is \textbf{ergodic} when the $\Phi$-closure of any non-null Borel set $A$ has full measure. 

The \textbf{full group} of a measure-preserving equivalence relation $\mathcal R$ is the group $[\mathcal R]$ of automorphisms of $(X,\mu)$ which induce permutations in the $\mathcal R$-classes, that is
$$[\mathcal R]=\{\varphi\in\Aut(X,\mu): \forall x\in X, \varphi(x)\,\mathcal R\, x\}.$$
It is a separable group when equipped with the complete metric $d_u$ defined by $$d_u(T,U)=\mu(\{x\in X: T(x)\neq U(x)\}.$$ The metric $d_u$ is called the \textbf{uniform metric} and the topology it induces is called the uniform topology. One also defines the \textbf{pseudo full group} of $\mathcal R$, denoted by $[[\mathcal R]]$, which consists of all partial isomorphisms $\varphi$ such that $\varphi(x)\, \mathcal R \, x$ for all $x\in \dom\varphi$.

Given a measure-preserving equivalence relation $\mathcal R$, we denote by $M_\mathcal R$ the $\sigma$-algebra of all Borel sets $A$ such that $T(A)=A$ for all $T\in[\mathcal R]$. Consider the Hilbert space of $\mathcal R$-invariant functions, i.e. the set of all $f\in\LL^2(X,\mu)$ such that for any $T\in [\mathcal R]$, one has $f\circ T=f$ a.e. We denote this closed Hilbert space by $\LL^2_{\mathcal R}(X,\mu)$ and remark that it consists of all the $M_{\mathcal R}$-measurable elements of $\LL^2(X,\mu)$. By definition $\mathcal R$ is \textbf{ergodic} if $\LL^2_{\mathcal R}(X,\mu)$ only consists of constant functions. A graphing $\Phi$ is ergodic if and only if $\mathcal R_\Phi$ is.

The orthogonal projection $\mathbb E_\mathcal R$ from the Hilbert space $\LL^2(X)$ onto the closed subspace $\LL^2_{\mathcal R}(X)$ satisfies the following equality, which defines it uniquely: for any $f\in \LL^2(X,\mu)$ and $g\in \LL^2_{\mathcal R}(X,\mu)$,
$$\int_X fg=\int_X\mathbb E_{\mathcal R}(f)g.$$
$\mathbb E_{\mathcal R}$ is called a \textbf{conditional expectation}. When $A$ is a subset of $X$, its characteristic function is an element of $\LL^2(X,\mu)$, and we call $\mathbb E_{\mathcal R}(\chi_A)$ the $\mathcal R$-\textbf{conditional measure} of $A$, denoted by $\mu_{\mathcal R}(A)$ or $\mu_\Phi(A)$ when $\mathcal R=\mathcal R_\Phi$. Often the equivalence relation $\mathcal R$ will be clear from the context and we will simply call  $\mu_{\mathcal R}(A)$ the conditional measure of $A$.  Because $\mathbb E_{\mathcal R}$ is a contraction for the $\LL^\infty$ norm, $\mu_{\mathcal R}(A)$ takes values in $[0,1]$. For those who are familiar with the ergodic decomposition, $\mu_{\mathcal R}(A)$ is just the function $x\mapsto \mu_x(A)$ where $(\mu_x)_{x\in X}$ is the ergodic decomposition of $\mathcal R$. Given $\varphi\in[[\mathcal R]]$, we have $\mu_{\mathcal R}(\dom\varphi)=\mu_{\mathcal R}(\rng\varphi)$. The following well-known result is a kind of converse (see \cite[Sec. 2.1]{lm14nonerg}). 

\begin{prop}[Dye]\label{transitive}Let $\mathcal R$ be a measure-preserving equivalence relation. Then if two sets $A$ and $B$ have the same $\mathcal R$-conditional measure, there exists $\varphi\in[[\mathcal R]]$ whose domain is $A$ and whose range is $B$.
\end{prop}

Note that in the ergodic case, the previous proposition implies that any two sets of the same measure can be mapped to each other by an element of the pseudo full group of $\mathcal R$. A measure-preserving equivalence relation is \textbf{aperiodic} if almost all its equivalence classes are infinite. The following proposition is well-known, see \cite[Sec. 2.1]{lm14nonerg} for a proof. 

\begin{prop}[\cite{MR0131516}, Maharam's lemma]\label{maha}
A measure-preserving equivalence relation $\mathcal R$ is aperiodic iff for any $A\subseteq X$, and for any $M_\mathcal R$-measurable function $f$ such that $0\leq f\leq \mathbb \mu_{\mathcal R}(A)$, there exists $B\subseteq A$ such that the $\mathcal R$-conditional measure of $B$ equals $f$.
\end{prop}

\begin{cor}\label{cor: cut a set in small subsets}
Let $\mathcal R$ be an aperiodic measure-preserving equivalence relation, and let $f:X\to [0,1]$ be an $M_\mathcal R$-measurable function. Let $A$ be a subset of $X$ such that for almost all $x\in A$, one has $f(x)>0$. Then there is a countable partition $(A_n)_{n\in\N}$ of $A$ such that for all $x\in X$ and all $n\in\N$, 
$$\mu_{\mathcal R}(A_n)(x)\leq f(x).$$
\end{cor}
\begin{proof}
The proof relies the following basic fact: there is a countable family of $M_{\mathcal R}$-measurable positive functions $f_n:X\to[0,1]$ such that for all $x\in X$, we have $f_n(x)\leq f(x)$ for all $n\in\N$  and 
$\sum_{n\in\N}f_n(x)=\mu_{\mathcal R}(x)$. Indeed, if for all $x\in X$ we let $f_n(x)=0$ if $\mu_\mathcal R(A)(x)=0$ and else
$$f_n(x)=\left\{\begin{array}{ll}f(x) & \text{if }n<\left\lfloor{\mu_\mathcal R(A)(x)}\right\rfloor \\
f(x)-\left\lfloor\frac{f(x)}{g(x)}\right\rfloor & \text{if }n=\left\lfloor{\mu_\mathcal R(A)(x)}\right\rfloor \\
0 & \text{else,}\end{array}\right.$$
it is easily checked that such functions satisfy the desired properties. We can now use the previous proposition to build by recurrence a sequence $(A_n)$ of disjoint subsets of $A$ such that for every $n\in\N$, $\mu_\mathcal R(A_n)=f_n\leq f$. Then $\mu_\mathcal R(A)=\sum_{n\in\N}\mu_\mathcal R(A_n)$ hence by integrating $\mu(A)=\sum_{n\in\N}\mu(A_n)=\mu(\bigsqcup_{n\in\N} A_n)$ so that $(A_n)$ is a partition of $A$ meeting all our requirements.
\end{proof}

Let us also record the following easy consequences of the previous results.

\begin{lem}\label{lem: basic involution construction}Let $\mathcal R$ be a measure-preserving equivalence relation. Then if two disjoint Borel sets $A$ and $B$ have the same $\mathcal R$-conditional measure, there exists an involution $U\in[\mathcal R]$ whose support is equal to  $A\sqcup B$ such that $U(A)=B$.
\end{lem}
\begin{proof}
By Proposition \ref{transitive} there exists $\varphi\in[[\mathcal R]]$ such that $\varphi(A)=B$. Then the involution $U\in[\mathcal R_\Phi]$ defined by 
$$U(x)=\left\{\begin{array}{cl}\varphi(x) & \text{if }x\in A \\\varphi\inv(x) & \text{if }x\in B \\x & \text{otherwise}\end{array}\right.$$
clearly works.
\end{proof}

\begin{lem}\label{lem: invol of any support}Let $\mathcal R$ be a measure-preserving aperiodic equivalence relation and let  $C$ be a Borel subset of $X$. Then there exists an involution $U\in[\mathcal R]$ such that $\supp U=C$.
\end{lem}
\begin{proof}By Maharam's Lemma (Lemma \ref{maha}), one can write $C=A\sqcup B$ where $\mu_{\mathcal R}(A)=\mu_{\mathcal R}(B)=\mu_{\mathcal R}(C)/2$. The previous lemma can then readily be applied. 
\end{proof}

\begin{rmq} In Section \ref{sec: chara supports}, we will characterise which sets arise as supports of involutions belonging to the $\LL^1$ full group of $T$, where $T$ is an ergodic transformation. 
\end{rmq}

\subsection{Entropy}

The \textbf{entropy} of a countable partition $(A_i)_{I\in I}$ of $(X,\mu)$ is the quantity 
$$H((A_i)_{i\in I}):=-\sum_{i\in I}\mu(A_i)\ln\mu(A_i)$$
where we use the convention $0\ln0=0$. It represents the average amount of information we get when we choose an element $x\in X$ at random and only learn in which piece $A_i$ of the partition $x$ is. The following statement is well-known, see e.g. \cite[Fact 1.1.4]{MR2809170}. 

\begin{lem}\label{lem: l1 partition has finite entropy}Let $(A_n)_{n\in\Z}$ be a partition of $X$ such that $\sum_{n\in\Z} \abs n\mu(A_n)<+\infty$. Then $H((A_n)_{n\in\Z})<+\infty$. 
\end{lem}
\begin{proof}
Some elementary calculus yields the inequality
$$-t\ln t\leq mt+e^{-m-1}$$
for all $t,m>0$. In particular for all $n\in\Z^*$ we have 
$$-\mu(A_n)\ln\mu(A_n)\leq \abs n\mu(A_n)+e^{-\abs n-1}.$$
Summing this over $n\in\Z^*$ and applying our hypothesis $\sum_{n\in\Z^*} \abs n\mu(A_n)<+\infty$, we obtain that $-\sum_{n\in \Z^*}\mu(A_n)\ln\mu(A_n)<+\infty$. So $H((A_n)_{n\in\N})<+\infty$ as wanted. 
\end{proof}

\begin{rmq}A result of Austin yields more generally that whenever $\Gamma$ is a finitely generated group, the cocycle partitions of elements of the $\LL^1$ full group of any free $\Gamma$-action have finite entropy (see Lemma 2.1 in \cite{Austin:2016rm}).\end{rmq}

Given two partitions $(A_i)$ and $(B_j)$ of $X$, we let $(A_i)\vee (B_j)$ be their \textbf{join}, i.e. the partition $(A_i\cap B_j)$. We have the following natural inequality (for a proof see e.g. \cite[1.6.10]{MR2809170}).
\begin{lem}\label{lem: entropy bound on the join}Let $(A_i)$ and $(B_j)$ be two partitions of $X$. Then
$$H((A_i)\vee (B_j))\leq H((A_i))+H((B_j)).$$
\end{lem}

Note that in particular, the join of finitely many partitions with finite entropy has finite entropy.

Lemma \ref{lem: entropy bound on the join} enables one to give the classical definition of the entropy of $T\in\Aut(X,\mu)$. First, given a partition $\mathcal Q$, we let 
$$h(T,\mathcal P):=\lim_{n\to+\infty} \frac 1nH\left(\bigvee_{i=0}^{n-1}T^i(\mathcal P)\right),$$
where the limit exists by subadditivity of the sequence $H(\bigvee_{i=0}^{n-1}T^i(\mathcal P))$ (cf. Lemma \ref{lem: entropy bound on the join}). 
The \textbf{entropy} of $T$ is 
$$h(T)=\sup_{\mathcal P\text{ partition}}h(T,\mathcal P).$$
A partition $\mathcal P$ is called \textbf{generating} if up to measure zero, the $\sigma$-algebra generated by the $T$-translates of $\mathcal P$ is equal to the Borel $\sigma$-algebra of $X$. A more convenient way to see this for us is in terms of the measure algebra of $(X,\mu)$. 

Recall that given a standard probability space $(X,\mu)$, its \textbf{measure algebra} is denoted by $\MAlg(X,\mu)$ and consists of all Borel subsets of $X$, two such sets being identified if their symmetric difference has measure zero. It is naturally endowed with a metric $d_\mu$ defined by $d_\mu(A,B)=\mu(A\bigtriangleup B)$ which is complete and separable (see \cite[40.A]{MR0033869}). Then a partition $\mathcal P$ is generating if and only if the algebra generated by the $T$-translates of $\mathcal P$ is dense in $\MAlg(X,\mu)$.

Kolmogorov-Sinai's theorem states that whenever $\mathcal P$ is a generating partition \textit{of finite entropy}, one has 
$$h(T)=h(T,\mathcal P)<+\infty$$
Conversely, Krieger's finite generator theorem states that every transformation of finite entropy admits a finite generating partition \cite{MR0259068} (moreover, there is such a partition whose cardinality $k$ satisfies the inequality $e^{h(T)}\leq k\leq e^{h(T)}+1$). Krieger then proved  that every such transformation is actually conjugate to a minimal uniquely ergodic subshift \cite{MR0393402}, which will be relevant to us when we connect topological full groups to $\LL^1$ full groups (see Section \ref{sec: main thms}).

\section{$\LL^1$ full groups of graphings}\label{sec:L1 full groups of graphings}

\subsection{Definition and first properties}

Let $\Phi$ be a graphing, and let $\mathcal R_\Phi$ be the measure-preserving equivalence relation it generates. Such a graphing induces a metric $d_\Phi$ on the orbits of $\mathcal R_\Phi$ by putting for all $(x,y)\in\mathcal R$
$$d_\Phi(x,y)=\min \{n\in\N:\exists \varphi_1,..., \varphi_n\in \Phi^{±1}\text{ such that }y=\varphi_n\cdots \varphi_1 (x)\}.$$
 The $\LL^1$ full group of $\Phi$, denoted by $[\Phi]_1$, is then defined by
 $$[\Phi]_1:=\{T\in [\mathcal R_\Phi]: x\mapsto d_\Phi(x,T(x))\text{ is integrable}\}.$$
The  triangle inequality implies that $[\Phi]_1$ is indeed a group, which we equip with a right-invariant metric $\tilde d_\Phi$ defined by 
$$\tilde d_\Phi(S,T):=\int_Xd_\Phi(S(x),T(x)).$$
Our first example is fundamental since the study of single measure-preserving transformations is a central subject in ergodic theory. 
\begin{exemple}\label{ex: Z actions}
Suppose we are given a measure-preserving aperiodic $\Z$-action, that is, an element $T\in\Aut(X,\mu)$, all whose orbits are infinite. Then $\Phi=\{T\}$ is a graphing, and we denote by $[T]_1$ the associated $\LL^1$ full group. Given $S\in[\mathcal R_T]$, for all $x\in X$ there is a unique $c_S(x)\in\Z$ such that $S(x)=T^{c_S(x)}(x)$. The map $x\mapsto c_S(x)$ is called the \textbf{cocycle} associated to $S$ and it completely determines $S$. The partition  $(\{x\in X: c_S(x)=n\})_{n\in\Z}$ is called the \textbf{cocycle partition} associated to $S$. 
 For all $x\in X$ and all $S,S'\in[\mathcal R_T]$, we have the \textbf{cocycle identity} 
$$c_{S'S}(x)=c_{S'}(S(x))+c_S(x).$$ 
The $\LL^1$ full group $[T]_1:=[\mathcal R_\Phi]_1$ consists of all the $S\in[\mathcal R_T]$ such that the corresponding cocycle $c_S$ is integrable. The metric $\tilde d_\Phi$ is then given by
\begin{align*}\tilde d_\Phi(S,S')&=\int_X\abs{c_{S}(x)-c_{S'}(x)}d\mu(x)\end{align*}
This fundamental example will be explored in Section \ref{sec: L1 full groups of Z actions}. \end{exemple}

\begin{exemple}\label{ex: full group of R}
Suppose $\Phi$ is a graphing, all whose connected components are complete graphs. Then $d_\Phi$ is the discrete metric: for all $(x,y)\in\mathcal R_\Phi$,  $$d_\Phi(x,y)=\left\{\begin{array}{cl}0 & \text{ if }x=y,  \\1 & \text{ else.} \end{array}\right.$$
Hence we have $[\Phi]_1=[\mathcal R_\Phi]$ and $\tilde d_\Phi$ is equal to the uniform metric $d_u$.
\end{exemple}

\begin{rmq}
Observe that given a graphing $\Phi$, we may view $\mathcal R_\Phi$ as a graphing all whose connected components are complete graphs. Furthermore, for every $(x,y)\in\mathcal R_\Phi$ we clearly have $d_{\mathcal R_\Phi}(x,y)\leq d_\Phi(x,y)$. By integrating, we deduce that
 for every $T,U\in[\Phi]_1$ the following inequality holds:
$$d_u(T,U)\leq \tilde d_\Phi(T,U).$$
\end{rmq}

\begin{exemple}
Let $\Gamma$ be a finitely generated group and consider a measure-preserving $\Gamma$-action on $(X,\mu)$. If $S$ is a finite generating set of $\Gamma$, it defines a graphing which generates $\mathcal R_\Gamma$, and we can consider the associated $\LL^1$ full group $[S]_1$. Observe that if $S'$ is another finite generating set, then the metrics $d_S$ and $d_{S'}$ are bilipschitz-equivalent on every orbit with a uniform constant. So the metrics $\tilde d_{S}$ and $\tilde d_{S'}$ are bilipschitz-equivalent and we have $[S]_1=[S']_1$. We conclude that we can actually define the $\LL^1$ full group of the $\Gamma$-action by $[\Gamma]_1=[S]_1$ since it does not depend on the choice of a finite generating set. 

Note that the $\LL^1$ full group of the action of a finitely generated group comes naturally equipped with a complete left invariant metric well-defined up to bilipschitz equivalence. Of course Example \ref{ex: Z actions} falls into this category with $\Gamma=\Z$. Moreover, if $\Gamma$ acts freely, every $T\in[\Gamma]$ is uniquely defined by the associated cocycle $c_T: X\to\Gamma$ defined by $T(x)=c_T(x)\cdot x$. Such a cocycle satisfies the cocycle identity $c_{T'T}(x)=c_{T'}(T(x))c_T(x)$. 
\end{exemple}

\begin{prop}\label{prop: complete separable metric}Let $\Phi$ be a graphing. 
The metric $\tilde d_\Phi$ is complete and separable.
\end{prop}
\begin{proof}
Let $(T_n)$ be a Cauchy sequence for $\tilde d_\Phi$. Note that since a Cauchy sequence converges iff it admits a converging subsequence we can freely replace $(T_n)$ by a subsequence. 

By the remark following Example \ref{ex: full group of R}, for all $T,T'\in [\Phi]_1$ we have $d_u(T,T')\leq \tilde d_\Phi(T,T')$, so the sequence $(T_n)$ is $d_u$-Cauchy and admits a $d_u$-limit $T\in[\mathcal R_\Phi]$. Up to taking a subsequence, we may assume that $T_n\to T$ pointwise and that for all $n\in\N$ we have $\tilde d_\Phi(T_n,T_{n+1})<\frac1{2^n}$. We deduce that 
$$\int_X\sum_{n\in\N} d_\Phi(T_n(x),T_{n+1}(x))d\mu(x)<2.$$
Now for all $n\in\N$ and all $x\in X$ we have \begin{align*}d_\Phi(T_n(x),T(x))&=\lim_{m\to+\infty}d_\Phi(T_n(x),T_m(x))\\&\leq\lim_{m\to+\infty} \sum_{i=n}^{m-1}d_\Phi(T_i(x),T_{i+1}(x))\\&\leq \sum_{i\in\N} d_\Phi(T_i(x),T_{i+1}(x))\end{align*}
By integrating this inequality for $n=0$ we deduce that $\tilde d_\Phi(T_0,T)<+\infty$ so that $T\in[\Phi]_1$.
Moreover this inequality allows us apply the Lebesgue dominated convergence theorem to the sequence of functions $x\mapsto d_\Phi(T_n(x),T(x))$ to conclude that $d_{\tilde\Phi}(T_n,T)\to 0$. The metric $\tilde d_\Phi$ is therefore complete. 

The separability follows from the line of ideas present in \cite[Thm. 3.17]{lcfullgroups}. Namely, let $S$ be a family of involutions such that
$$\bigcup_{s\in S}\{(x,s(x)): x\in X\}= \mathcal G_\Phi\cup\{(x,x): x\in X\}$$
and define $\Gamma=\la S\ra$ equipped with the word metric $d_S$. We consider the space $\LL^1(X,\mu,\Gamma)$ of measurable functions taking values in $\Gamma$ such that $\int_Xd_S(f(x),1_\Gamma)<+\infty$. Such a space is separable for the natural $\LL^1$ metric, and $[\Phi]_1$ is a continuous quotient of a subspace of the metric space $\LL^1(X,\mu,\Gamma)$ (namely, the subspace of maps $f\in\LL^1(X,\mu,\Gamma)$ such that $T_f: x\mapsto f(x)\cdot x$ belongs to $\Aut(X,\mu)$), hence separable. 
\end{proof}

\begin{prop}
Let $\Phi$ be a graphing. The topology induced by the metric $\tilde d_\Phi$ is a group topology.
\end{prop}
\begin{proof}
Since the metric $\tilde d_\Phi$ is right-invariant it suffices to show that left multiplication is continuous. 
Let $V\in [\Phi]_1$ and suppose $\tilde d_\Phi(U_n,U)\to 0$. In particular $d_u(U_n,U)\to 0$ so if we let $A_n=\{x\in X: U_n(x)\neq U(x)\}$, then $\mu(A_n)\to 0$ and we have
\begin{align*}
\tilde d_\Phi(VU_n, VU)&=\int_{A_n} d_\Phi(VU_n(x),VU(x))\\
&\leq \int_{A_n}d_\Phi(VU_n(x),U_n(x))+\int_{A_n}d_\Phi(U_n(x),U(x))+\int_{A_n}d_\Phi(U(x),VU(x))\\
&\leq \int_{U_n(A_n)}d_\Phi(V(x),x)+\tilde d_\Phi(U_n,U)+\int_{A_n}d_\Phi(U(x),VU(x)).
\end{align*}
Now since $\mu(A_n)\to 0$ the right term converges to zero by the Lebesgue dominated convergence theorem. Since $U_n$ is measure-preserving we also have $\mu(U_n(A_n))\to 0$ so the left term also converges to zero. By assumption the middle term converges to zero, so $\tilde d_\Phi(VU_n, VU)\to 0$ as wanted. 
\end{proof}

\begin{rmq}
Because $\tilde d_\Phi$ is right-invariant, the two previous propositions show that $[\Phi]_1$ is a \textbf{cli} Polish group, which means it admits a complete left (or equivalently right)-invariant  metric compatible with its topology. Examples of cli Polish groups include all the locally compact second-countable groups (see \cite[Prop. 3.C.2]{MR1478843}), and non-examples include the following Polish groups:  $\Aut(X,\mu)$, the group of unitaries of an infinite dimensional Hilbert space and the group of permutation of the integers. 

Using Rohlin's lemma, it can be shown that whenever $T\in\Aut(X,\mu)$ is aperiodic, the $\LL^1$ full group $[T]_1$ is never SIN (meaning that its topology does not admit a basis at the identity made of conjugacy-invariant neighborhoods). To our knowledge, such $\LL^1$ full groups are the first examples of cli Polish groups which are non-SIN but at the same time far from being locally compact groups. 
\end{rmq}

A fundamental feature of $\LL^1$ full groups is that they are stable under taking induced maps. 

\begin{prop}\label{prop: induced transfo}Let $\Phi$ be a graphing, let $[\Phi]_1$ be the associated $\LL^1$-full group. Then for all $T\in[\Phi]_1$, we have $T_A\in[\Phi]_1$, and 
$$\tilde d_\Phi(T_A,\mathrm{id}_X)\leq \tilde d_\Phi(T,\mathrm{id}_X).$$
\end{prop}
\begin{proof}
For all $n\in\N$, let $A_n:=\{x\in X: T_A(x)=T^n(x)\}$. Then the $A_n$'s form a partition of $A$. For all $n\in\N$ and all $0\leq m<n$, we let $B_{n,m}:=T^m(A_n)$. The family $(B_{n,m})_{0\leq m<n}$ is a collection of disjoint sets.

We now fix $n\in\N$. By the triangle inequality, we have that for all $x\in A_n$,

$$d_\Phi(x,T_A(x))=d_\Phi(x,T^n(x))\leq \sum_{m=0}^nd_\Phi(T^m(x), T^{m+1}(x)).$$
By integrating and summing this over all $n\in\N$, we deduce that

$$\int_A d_\Phi(x,T_A(x))=\sum_{n\in\N}\sum_{m=0}^n \int_{B_{n,m}}d_\Phi(x,T(x))\leq \int_Xd_\Phi(x,T(x)).$$
To conclude, we note that the left term above is equal to $\tilde d_\Phi(T_A,\mathrm{id}_X)$, while the right term is equal to $\tilde d_\Phi(T,\mathrm{id_X})$.
\end{proof}
\begin{rmq}The above proof is basically the same as that of Kac's formula \cite[Thm.2']{MR0022323}. We will use this proof again to get later a similar statement in the case of $\LL^1$ full groups of $\Z$-actions (see Prop. \ref{prop:indexinduced}). 
\end{rmq}

The following lemma will provide us many involutions in $\LL^1$ full groups. 

\begin{lem}\label{manyinvol}
Let $\Phi$ be a graphing. Then for every involution $U\in[\mathcal R_\Phi]$, there exists an increasing sequence of $U$-invariant sets $A_n\subseteq \supp U$ such that $\supp U=\bigcup_{n\in\N} A_n$ and for all $n\in\N$, $U_{A_n}\in[\Phi]_1$.
\end{lem}
\begin{proof}
Let $U\in[\mathcal R_\Phi]$ be an involution and for all $n\in\N$, let $A_n=\{x\in \supp U: d_\Phi(x,U(x))<n\}$. Since $U$ is an involution and $d_\Phi$ is symmetric each $A_n$ is $U$-invariant, and $\bigcup_{n\in\N} A_n=\supp U$, so $\mu(\supp U\setminus A_n)\to 0$. By the definition of $A_n$, each $U_{A_n}$ belongs to $[\Phi]_1$.
\end{proof}

\begin{thm}\label{thm: L1 is dense}
Let $\Phi$ be a graphing. Then $[\Phi]_1$ is dense in $[\mathcal R_\Phi]$ for the uniform topology.
\end{thm}
\begin{proof}
Because full groups are generated by involutions (see \cite[Lem. 4.5]{MR2583950}), we only need to show that every involution in $[\mathcal R_\Phi]$ can be approximated by elements of $[\Phi]_1$. Let $U\in[\mathcal R_\Phi]$ be such an involution. 

By the previous lemma we have an increasing sequence $(A_n)$ of $U$-invariant subsets of $\supp U$ such that $\bigcup_{n\in\N} A_n=\supp U$ and for all $n\in\N$ the involution $U_{A_n}$ belongs to $[\Phi]_1$. So $\mu(\supp U\setminus A_n)\to 0$ and we conclude that $d_u(U_{A_n},U)\to 0$ as desired.
\end{proof}

\subsection{The closure of the derived group is topologically generated by involutions}\label{sec: gen closure of derived group with involutions}

As we will see later, given a graphing $\Phi$ and $T\in[\Phi]_1$, the map $A\in\MAlg(X,\mu)\mapsto T_A\in[\Phi]_1$ is not continuous in general. We however have the following very useful statement.

\begin{prop}\label{prop: weak order continuity}
Let $A\subseteq X$, let $\Phi$ be a graphing and take $T\in[\Phi]_1$. If $(A_n)$ is a sequence of subsets of $A$ such that $A=\bigcup_{n\in\N} A_n$ then 
$$T_A=\lim_{n\to+\infty}T_{A_n}.$$
\end{prop}
\begin{proof}
Let $k:X\to \N$ be defined by $T_A(x)=T^{k(x)}(x)$ and similarly for all $n\in\N$  define $k_n:X\to\N$ by $T_{A_n}(x)=T^{k_n(x)}(x)$. Since $\mu(A_n\bigtriangleup A)\to 0$, we have that $k_n\to k$ pointwise. So if we let $B_n=\{x\in X: k(x)\neq k_n(x)\}$, we have $\mu(B_n)\to 0$. By the definition of $B_n$ we get
\begin{align*}\tilde d_\Phi(T_{A_n},T_A)
&=\int_{B_n} d_\Phi(T_{A_n}(x),T_A(x))\end{align*}
The triangle inequality now yields
\begin{align}\label{ineq:bound the distance}
\tilde d_\Phi(T_{A_n},T_A)&\leq \int_{B_n}d_\Phi(T_{A_n}(x),x)+\int_{B_n}d_\Phi(T_A(x),x).
\end{align}

Since $A_n\subseteq A$ we have $(T_A)_{A_n}=T_{A_n}$, so Proposition \ref{prop: induced transfo} yields 
$$\tilde d_\Phi(T_{A_n},\mathrm{id_X})\leq \tilde d_\Phi(T_A,\mathrm{id_X}).$$ By developing the left-hand term we get 
$$\int_{B_n}d_\Phi(T_{A_n}(x),x)+\int_{X\setminus B_n}d_\Phi(T_{A_n}(x),x)\leq\int_X\tilde d_\Phi(T_A(x),x)).$$
Since for all $x\in X\setminus B_n$ we have $T_A(x)=T_{A_n}(x)$, we may rewrite this as
\begin{align*}
\int_{B_n}d_\Phi(T_{A_n}(x),x)&\leq\int_{X}d_\Phi(T_A(x),x)-\int_{X\setminus B_n}d_\Phi(T_A(x),x)\\
\int_{B_n}d_\Phi(T_{A_n}(x),x)&\leq \int_{B_n}d_\Phi(T_A(x),x).
\end{align*}
We now reinject this inequality into inequality (\ref{ineq:bound the distance}) and obtain
\begin{align*}
\tilde d_\Phi(T_{A_n},T_A)&\leq 2\int_{B_n}d_\Phi(T_{A}(x),x).
\end{align*}
By the dominated convergence theorem we now have $\int_{B_n}d_\Phi(T_{A}(x),x)\to 0$ so we  conclude that $d_\Phi(T_{A_n},T_A)\to 0$ as desired.
\end{proof}

The above proposition is a direct consequence of the dominated convergence theorem when the sets $A_n$ are already $T_A$-invariant. We will often use it in this easier case, but the general statement will appear at a crucial point towards proving that the closure of the derived group of $[\Phi]_1$ is topologically generated by involutions. Let us for the moment make sure that every involution belongs to the closure of the derived group. 

\begin{lem}\label{commuinvol}Suppose $\Phi$ is aperiodic. Then every involution in $[\Phi]_1$ belongs to the closure of derived group of $[\Phi]_1$.
\end{lem}

\begin{proof}
Let $T\in [\Phi]_1$ be an involution. We can find $B\subseteq X$ such that $\supp T= B\sqcup T(B)$. By aperiodicity, there is an involution $\tilde U \in [\mathcal R_\Phi]$ such that $\supp \tilde U=B$ (see Lemma \ref{lem: invol of any support}). We then find $C\subseteq B$ such that $B=C\sqcup \tilde U(C)$. Lemma \ref{manyinvol} then yields an increasing sequence of $\tilde U$-invariant sets $B_n$ such that $B=\bigcup_{n\in\N} B_n$, and for all $n\in\N$, $\tilde U_{B_n}\in [\Phi]_1$. We let $C_n=C\cap B_n$, and $T_n:=T_{C_n\sqcup T(C_n)}$. 

Now we define an involution $U_n$ in the $\LL^1$ full group of $\Phi$ by 
$$U_n(x):=\left\{\begin{array}{cl}\tilde U(x) & \text{if }x\in{B_n\sqcup U(B_n)}, \\T\tilde UT(x) & \text{if }x\in{T(B_n\sqcup U(B_n))}, \\x & \text{else.}\end{array}\right.$$
By construction, the commutator $[T_n,U_n]$ is the tranformation induced by $T$ on the $T$-invariant set $A_n:=B_n\sqcup T(B_n)$, and we have $\bigcup_{n\in\N} A_n=\supp T$. We conclude by the previous proposition that $[T_n,U_n]\to T$ as desired. 
\end{proof}

\begin{rmq}I do not know wether involutions actually belong to the derived group itself. As a more general question, when does it happen that $D([\Phi]_1)$ is closed? Note that full groups of measure-preserving equivalence relations (which are a special case of $\LL^1$ full groups by \ref{ex: full group of R}) are perfect if and only if the equivalence relation is aperiodic (see \cite[Prop. 3.6]{lemai2014}). 
\end{rmq}

We will now gradually establish that conversely, the closure of $D([\Phi]_1)$ is topologically generated by involutions. 

\begin{lem}\label{lem: periodic in derived group} Let $\Phi$ be an aperiodic graphing. Then every periodic element of $[\Phi]_1$ belongs to the closed subgroup of $[\Phi]_1$ generated by involutions and hence to the closure of the derived group of $[\Phi]_1$.
\end{lem}
\begin{proof}
Let $U\in [\Phi]_1$ be periodic, and for all $n\in\N^*$ let $A_n$ denote the $U$-invariant set of all $x\in X$ whose $U$-orbit has cardinality $n$. By Proposition \ref{prop: weak order continuity}, we have $$U=\lim_{n\to+\infty}U_{\bigcup_{i=1}^nA_i}=\lim_{n\to+\infty}\prod_{i=2}^n U_{A_i}.$$
so it suffices to prove the statement for cycles, i.e. elements $U$ of $[\Phi]_1$ for which there is $n\in\N$ such that every non-trivial $U$-orbit has cardinality $n$. If $U\in[\Phi]_1$ is such an element, we may find $A\subseteq X$ such that 
$$\supp U=\bigsqcup_{i=1}^{n}U ^i(A).$$
For $i\in\{1,...,n-1\}$ we  let $U_i\in[\Phi]_1$ be the involution defined by 
$$U_i(x)=\left\{\begin{array}{cl}U(x) & \text{if }x\in U ^{i}(A) \\U\inv(x) & \text{if }x\in U^{i+1}(A) \\x & \text{otherwise.}\end{array}\right.$$
It now follows from the well-known identity in symmetric groups $$(1\quad 2\;\cdots\; n)=(1\quad 2)(2\quad 3)\cdots(n-1\quad n)$$ that $U=U_1U_2\cdots U_{n-1}$, which proves the first part of the lemma. The ‘‘as well as'' part follows from Lemma \ref{commuinvol}.
\end{proof}

\begin{lem}Let $\Phi$ be an aperiodic graphing. The $\LL^1$ full group $[\Phi]_1$ is generated by elements whose support has $\Phi$-conditional measure everywhere less than $1/4$. 
\end{lem}
\begin{rmq}As the proof shows, $1/4$ can be replaced by any $\epsilon>0$, but we won't need that.
\end{rmq}

\begin{proof}
First note that by breaking a fundamental domain into pieces of small conditional measure (using Maharam's lemma \ref{maha}), one can show every periodic element is the product of elements whose support has $\Phi$-conditional measure everywhere less than $1/4$. 

Now if $T\in[\Phi]_1$, one may write $T=T_pT'$ where $T_p$ is periodic and all the non-trivial $T'$-orbits are infinite. By Maharam's lemma we then find $A\subseteq \supp T'$  such that for all $x\in \supp T'$, we have $$0<\mu_{T'}(A)(x)<1/4.$$
Then $A$ intersects almost every non-trivial $T'$-orbit so that $T'{T'}_A\inv$ is periodic. Since $A$ has $T'$-conditional measure everywhere less than $1/4$, the same is true of its $\Phi$-conditional measure. So the support of ${T'}_A$ has $\Phi$-conditional measure everywhere less than $1/4$ as wanted, and $T=T_p(T'{T'}_A\inv){T'}_A$ can be written as a product of elements whose support has $\Phi$-conditional measure everywhere less than $1/4$. 
\end{proof}

\begin{thm}\label{thm: invol gen}Let $\Phi$ be an aperiodic graphing. Then the closure of the derived group of $[\Phi]_1$ is topologically generated by involutions.
\end{thm}
\begin{proof}
Denote by $\mathcal S$ the set of $T\in[\Phi]_1$ whose support has $\Phi$-conditional measure everywhere less than $1/4$. By the previous lemma $\mathcal S$ generates $[\Phi]_1$. Since $\mathcal S$ is moreover invariant under conjugacy, it suffices to show that the commutator of any two elements of $S$ belongs to the closure of the subgroup generated by involutions. 

So let $T,U\in \mathcal S$ and let $B=\supp T\cup \supp U$. There exists an involution $V\in[\mathcal R_\Phi]$ such that $B$ and $V(B)$ are disjoint and $\supp V=B\sqcup V(B)=:A$. By Lemma \ref{manyinvol}
there exists an increasing sequence $(A_n)$ of $V$-invariant subsets of $A$ such that $V_{A_n}\in[\Phi]_1$ and $\bigcup A_n=A$. 

By Proposition \ref{prop: weak order continuity} we have $T_{A_n}\to T$ and $U_{A_n}\to U$, so it suffices to show that for every $n\in\N$, the commutator $[T_{A_n},U_{A_n}]$ belongs to the closed group generated by involutions. 

To this end, we fix $n\in\N$ and observe that the set $\supp T_{A_n}\cup \supp U_{A_n}$ is disjoint from 
$V_{A_n}(\supp T_{A_n}\cup \supp U_{A_n})$. It follows that if we let $T'_n=V_{A_n}T_{A_n}\inv V_{A_n}$, then $T'_n$ commutes with both $T_{A_n}$ and $U_{A_n}$. So we have 
\begin{align*}[T_{A_n},U_{A_n}]&=[T_{A_n}T'_n,U_{A_n}]\\
&=[(T_{A_n}V_{A_n}T_{A_n}\inv)V_{A_n},U_{A_n}]
\end{align*} 
and since taking a commutator means multiplying by a conjugate of the inverse, this equality implies $[T_{A_n},U_{A_n}]$ is a product of involutions in $[\Phi]_1$. As explained before $[T_{A_n},U_{A_n}]\to[T,U]$ so this ends the proof. 
\end{proof}

As a consequence, we can easily derive that the closure of the derived group is connected. 

\begin{cor}\label{cor: derived group is connected}Let $\Phi$ be an aperiodic graphing. Then $\overline{D([\Phi]_1)}$ is connected.
\end{cor}
\begin{proof}
By the previous theorem it suffices to show that the group generated by involutions in $[\Phi]_1$ is connected, and we will actually prove that it is path connected. Let $U\in[\Phi]_1$ be an involution and let $A$ be a fundamental domain for $U$. There is an increasing family of Borel subsets $(A_t)_{t\in[0,1]}$ such that $A_{0}=\emptyset$ and $A_1=A$ and for all $t\in[0,1]$, $\mu(A_t)=t\mu(A)$. 

Then by the dominated convergence theorem the family of involutions $U_t:=U_{A_t\cup U(A_t)}$ is a continuous path from $U_0=\mathrm{id_X}$ to $U_1=U$. So every involution in $[\Phi]_1$ is connected by a path to the identity, and we conclude that the group generated by involutions is path connected as desired. 
\end{proof}

For ergodic $\Z$-actions, $\overline{D([T]_1)}$ is actually the connected component of the identity (this follows from Corollary \ref{cor:index take values in Z} along with the above result). \\

We end this section by refining the fact that the closure of the derived group is topologically generated by involutions so as to obtain the same statements that Kittrell and Tsankov had obtained for full groups in \cite[Sec. 4.2]{MR2599891}. The proof is actually the same as Kittrell and Tsankov's modulo the use of the dominated convergence theorem and we reproduce it here for the reader's convenience. 

Let $\Phi$ be a graphing and $\varphi\in\Phi$. If $A$ is a Borel subset of $\dom\varphi$ such that $A\cap \varphi(A)=\emptyset$, we denote by $I_{\varphi,A}$ the involution in $\Aut(X,\mu)$ defined by: for all $x\in X$,
$$I_{\varphi,A}(x)=
\left\{\begin{array}{cl}\varphi(x) & \text{if }x\in A \\\varphi\inv(x) & \text{if }x\in \varphi(A) \\x & \text{otherwise.}\end{array}\right.$$
Note that $I_{\varphi,A}\in [\Phi]_1$ since $d_\Phi(x,I_{\varphi,A}(x))\leq 1$ for all $x\in X$. 

\begin{thm}\label{thm: topo gen by involutions induced by graphing}Let $\Phi$ be an aperiodic graphing. Then the group $\overline{D([\Phi]_1)}$ is topologically generated by the set
$$\left\{I_{\varphi,A}: \varphi\in\Phi, A\subseteq \dom\varphi\text{ and } A\cap \varphi(A)=\emptyset\right\}.$$
\end{thm}
\begin{proof}
By Theorem \ref{thm: invol gen}, we only need to show that involutions belong to the group topologically generated by the set 
$$\left\{I_{\varphi,A}: \varphi\in\Phi, A\subseteq \dom\varphi\text{ and } A\cap \varphi(A)=\emptyset\right\}.$$
So let $U\in[\Phi]_1$ be an involution, and let $A$ be a fundamental domain for the restriction of $U$ to its support. Let $(w_n)$ enumerate the $\Phi$-words, and for each $n\in\N$ let 
\begin{align*}
A_n=&\{x\in A: U(x)=w_n(x)\text{ and for all }m\in\N\text{ if }U(x)=w_m(x) \\ &\text{ then either }\abs{w_m}\geq \abs{w_n}\text{ or }m\geq n\}
\end{align*}
Then $(A_n)$ forms a partition of $A$ and for all $n\in\N$, if $w_n=\varphi_1\cdots\varphi_k$, then for all $x\in A_n$  we have that for all $i\neq j\in\{1,...,k\}$ $$\varphi_i\cdots \varphi_k(x)\neq \varphi_j\cdots\varphi_k(x)$$ by minimality of the length of $w_n$. By the dominated convergence theorem, 
$$U=\lim_{N\to+\infty}\prod_{n=0}^N I_{w_n,A_n},$$
so we may actually assume that $U=I_{\varphi_1\cdots\varphi_k,B}$ for some $B\subseteq X$ and $\varphi_1,...,\varphi_k\in\Phi^{\pm 1}$ such that
for all $i\neq j\in\{1,...,k\}$ and all $x\in B$, $\varphi_i\cdots \varphi_k(x)\neq \varphi_j\cdots\varphi_k(x)$.

 Since $X$ is a standard Borel space, a well-known argument yields that we may actually find a partition $(B_{m})_{m\in\N}$ of $B$ such that for all $m\in\N$ and all $i\neq j\in\{1,...,k\}$, 
$$\varphi_i\cdots \varphi_k(B_m)\cap \varphi_j\cdots\varphi_k(B_m)=\emptyset$$
(see for instance \cite[Lem. 5.2]{Eisenmann:2014oq}). 
Again by the dominated convergence theorem 
$$U=\lim_{N\to+\infty}\prod_{n=0}^N I_{\varphi_1\cdots\varphi_k,B_m},$$
so we may actually assume that $U=I_{\varphi_1\cdots\varphi_k,C}$ for some $C\subseteq X$ such that for  all $i\neq j\in\{1,...,k\}$, 
$$\varphi_i\cdots \varphi_k(C)\cap \varphi_j\cdots\varphi_k(C)=\emptyset$$
Then note that for all $m\in\N$ and $i\in\{1,...,k-1\}$, the identity $$(i-1\quad k)=(i-1\quad i)(i\quad k)(i-1\quad i)$$ in the symmetric group over $\{0,...,k\}$ yields 
$$I_{\varphi_i\cdots\varphi_k,C}=I_{\varphi_i, \varphi_{i+1}\cdots \varphi_k(C)}I_{\varphi_{i+1}\cdots\varphi_k,C}I_{\varphi_i, \varphi_{i+1}\cdots \varphi_k(C)}$$
so that by a downward recurrence on $i\in\{1,...,k\}$ we have that $I_{\varphi_i\cdots\varphi_k,C}$ belongs to the group generated by 
$$\left\{I_{\varphi,A}: \varphi\in\Phi, A\subseteq \dom\varphi\text{ and } A\cap \varphi(A)=\emptyset\right\}.$$
In particular, $U=I_{\varphi_1\cdots\varphi_k,C}$ belongs to such a group, as wanted.
\end{proof}

\subsection{Dye's reconstruction theorem}

We now show that $\LL^1$ full groups as abstract groups can only be isomorphic through a measure-preserving transformation, which is a key step in order to see than $\LL^1$ full groups of ergodic $\Z$-actions are a complete invariant of flip conjugacy. Fortunately, the main result that we need for this has already been proven by Fremlin in a much more general context. 

\begin{df}[Fremlin] A subgroup $G\leq \Aut(X,\mu)$ \textbf{has many involutions} if for every Borel $A\subseteq X$ non-null, there exists a non-trivial involution $U\in G$ such that $\supp U\subseteq A$. 
\end{df}

\begin{prop}\label{prop: chara many involutions}
Let $\Phi$ be a graphing. The following are equivalent
\begin{enumerate}[(1)]
\item $\Phi$ is aperiodic;
\item $[\Phi]_1$ has many involutions;
\item $\overline{D([\Phi]_1)}$ has many involutions;
\item $D([\Phi]_1)$ has many involutions.
\end{enumerate}
\end{prop}
\begin{proof}
Note that (2) and (3) are equivalent since every involution in $[\Phi]_1$ belongs to $\overline{D([\Phi]_1)}$ by Lemma \ref{lem: periodic in derived group}. 

Let us show by contrapositive that (2)$\impl$(1): if $\Phi$ is not aperiodic, then we may find a non-null set $A$ which intersects every $\Phi$-orbit in at most one point. Clearly every element of  $[\Phi]_1$ (actually of $[\mathcal R_\Phi]$) supported in $A$ is trivial, in particular $[\Phi]_1$ does not have many involutions.

For the converse (1)$\impl$(2), suppose that $\Phi$ is aperiodic and let $A\subseteq X$ be non-null. Then by Lemma \ref{lem: invol of any support}, there exists an involution $U\in[\mathcal R_\Phi]$ whose support is equal to $A$. Now Lemma \ref{manyinvol} ensures us that there is a non-null $U$-invariant set $A'$ such that $U_{A'}\in [\Phi]_1$, witnessing that $[\Phi]_1$ has many involutions. So the implication (1)$\impl$(2) also holds and we conclude that (1), (2) and (3) are equivalent.

Clearly $(4)\impl(2)$, so we now only need to show that (2)$\impl$(4). Suppose that $[\Phi]_1$ has many involutions and let $A$ be a non-null Borel subset of $X$. Let $U\in[\Phi]_1$ be an involution supported in $A$, let $B$ be a fundamental domain of the restriction of $U$ to its support. Finally, let $V\in[\Phi]_1$ be an involution supported in $B$. Since $B$ and $U(B)$ are disjoint, $UVUV$ is an involution supported in $A$ and  belonging to $D([\Phi]_1)$ as wanted: $D([\Phi]_1)$ has many involutions.
\end{proof}

We denote by $\Aut^*(X,\mu)$ the group of non-singular transformations of $(X,\mu)$, i.e. Borel bijections $T:X\to X$ such that for all Borel $A\subseteq X$, one has $\mu(A)=0$ if and only if $\mu(T(A))=0$. 

We can now state and use Fremlin's theorem, which is a generalization of Dye's reconstruction theorem for full groups. 

\begin{thm}[{Fremlin \cite[384D]{MR2459668}}]Let $G,H\leq \Aut(X,\mu)$ be two groups with many involutions. Then any isomorphism between $G$ and $H$ is the conjugacy by some non-singular transformation: for all $\psi:G\to H$ group isomorphism, there is $S\in\Aut^*(X,\mu)$ such that for all $T\in G$, 
$$\psi(T)=STS\inv.$$
\end{thm}

\begin{cor}\label{cor: Dye reconstruction for L1}
Let $\Phi$ and $\Psi$ be two ergodic graphings. Then every isomorphism between $[\Phi]_1$ and $[\Psi]_1$ is the conjugacy by some measure-preserving transformation.  
\end{cor}
\begin{proof}
By Proposition \ref{prop: chara many involutions} and the previous theorem, if $\psi:[\Phi]_1\to [\Psi]_1$ is a group isomorphism then there is $S\in\Aut^*(X,\mu)$ such that for all $T\in [\Phi]_1$ one has 
$$\psi(T)=STS\inv.$$
Let $f$ be the Radon-Nikodym derivative of the probability measure $S_*\mu$ with respect to $\mu$. Since $\mu$ is preserved by every $T\in[\Phi]_1$, the measure $S_*\mu$ is preserved by every $T'\in [\Psi]_1$. So $f\circ T'=f$ a.e. for every $T'\in[\Psi]_1$. But $\Psi$ is ergodic so this implies $f$ is constant: the measure $\mu$ is preserved by $S$ as wanted.
\end{proof}

\subsection{Closed normal subgroups of the closure of the derived group}\label{sec:closed normal subgroups}

 Given a group $G\subseteq \Aut(X,\mu)$, we denote by $G_A$ the subgroup of all $T\in G$ such that $\supp T\subseteq A$. Note that if $\Phi$ is a graphing and $A$ is $\Phi$-invariant then $\overline{D([\Phi]_1)}_A$ is a closed normal subgroup of $\overline{D([\Phi]_1)}$. 
 
 We will see that conversely if $\Phi$ is aperiodic then all the closed normal subgroups of $\overline{D([\Phi]_1)}$ arise in this manner. Our approach is based on the study of closed normal subgroups generated by involutions.

\begin{lem}\label{lem: restriction of derived group}Let $\Phi$ be a graphing. If $A$ is a $\Phi$-invariant set then $$\overline{D({[\Phi]_1}_A)}=\overline{(D([\Phi]_1)_A)}=(\overline{D([\Phi]_1)})_A.$$
\end{lem}
\begin{proof}
To simplify notation we let $G=[\Phi_1]$, so that we aim to prove that 
$$\overline{D(G_A)}=\overline{(D(G)_A)}=(\overline{D(G)})_A.$$
To this end we show that the chain of inclusions $\overline{D(G_A)}\subseteq\overline{(D(G)_A)}\subseteq(\overline{D(G)})_A\subseteq\overline{D(G_A)}$ holds. 
\begin{itemize}
\item We clearly have the inclusion $D(G_A)\subseteq D(G)_A$ so that $\overline{D(G_A)}\subseteq\overline{(D(G)_A)}$. 
\item Next, note that $(\overline{D(G)})_A$ is a closed subgroup of $G$ which contains every element of $D(G)_A$. It follows that $\overline{(D(G)_A)}\subseteq (\overline{D(G)})_A$. 
\item For the last remaining inclusion $(\overline{D(G)})_A\subseteq\overline{D(G_A)}$, let $T\in (\overline{D(G)})_A$. Then $T$ is supported on $A$ and is a limit of elements $T_n$ belonging to $D(G)$. Now note that by definition of the topology $T$ is also the limit of the sequence of transformations $({T_n}_A)_{n\in\N}$. Moreover each ${T_n}_A$ can be rewritten as a product of commutators of elements supported in $A$ by taking the induced transformations instead, which is harmless since $A$ is invariant by every element appearing in such a product. We conclude that  $T\in \overline{D(G_A)}$ as desired.\qedhere
\end{itemize}
\end{proof}

We now establish a crucial lemma allowing a better understanding of closures of conjugacy classes of involutions.

\begin{lem}\label{lem: conjugate involutions}Let $\Phi$ be a graphing. Let $U$ and $V$ be two involutions such that their supports are disjoint and have the same $\Phi$-conditional measure. Then $V$ belongs to the closure of the conjugacy class of  $U$. 
\end{lem}
\begin{proof}
Let $A$ and $B$ be fundamental domains of the respective restrictions of $U$ and $V$ to their support. Then $\mu_\Phi(A)=\mu_\Phi(B)$ so there exists an involution $T\in[\mathcal R_\Phi]$ such that $T(A)=B$ and we can moreover assume that $T(x)=x$ for all $x\in A\cap B$. 

For all $n\in\N$ we let $A_n=\{x\in A: d_\Phi(x,T(x))<n\}$ and $B_n=\{x\in B: d_\Phi(x,T(x))<n\}$. Since $T$ is an involution and $T(A)=B$, we also have $B_n=T(A_n)$. Note that $(A_n)$ and $(B_n)$ are increasing sequences of sets and that $\bigcup_nA_n=A$, $\bigcup_nB_n=B$. We now define a new involution $T'_n$
$$T'_n(x)=\left\{\begin{array}{cl}T(x) & \text{if }x\in A_n\sqcup B_n \\VTU(x) & \text{if } x\in U(A_n) \\
UTV(x)&\text{if }x\in V(B_n)
\\x & \text{otherwise.}\end{array}\right.$$
The definition of $A_n$ and the fact that $V,U\in[\Phi]_1$ yield that $T'_n\in[\Phi]_1$ as well. For all $n\in\N$ and all $x\in X$, an easy calculation yields that:
\begin{itemize}
\item if $x\in (A\cup U(A))\setminus (A_n\cup U(A_n))$, then $T'_nUT'_n(x)=U(x)$, 
\item else if $x\in (B_n\cup V(B_n)$, then $T'_nUT'_n(x)=V(x)$ and
\item else $T'_nUT'_n(x)=x$. 
\end{itemize}
So $T'_nUT'_n\to V$ pointwise. Moreover the trichotomy above also shows that for all $x\in X$ we have  $d_\Phi(x, T'_nUT'_n(x))\leq d_\Phi(x,U(x))+d_\Phi(x,V(x))$, so by the dominated convergence theorem $\tilde d_\Phi(T'_nUT'_n,V)\to 0$. We have therefore established that $V$ belongs to the closure of the conjugacy class of $U$.
\end{proof}

\begin{lem}\label{lem: closed normal gen by invol}Let $\Phi$ be an aperiodic graphing,  let $U\in[\Phi]_1$ be an involution such that $\mu_\Phi(\supp U)\leq 1/2$. Let $A$ be the $\Phi$-closure of the support of $U$. Then the closure of the subgroup generated by conjugates of $U$ by elements of $\overline{D([\Phi]_1)}$ contains $\overline{D([\Phi]_1)}_A$.
\end{lem}
\begin{proof}

By Lemma \ref{lem: restriction of derived group} we may as well assume that $A=X$ so that our aim becomes to show that whenever $U$ has support whose $\Phi$-closure is equal to $X$ and which satisfies $\mu_\Phi(\supp U)\leq 1/2$, the closure of the subgroup generated by conjugates of $U$ by involutions contains $\overline{D([\Phi]_1)}$.

So let $U$ be as such, let $B=\supp U$ and denote by $G$ the closure of the subgroup generated by conjugates of $U$ by elements of $\overline{D([\Phi]_1)}$. Since $\mu_\Phi(\supp U)\leq 1/2$ by Maharam's lemma we may find a set $C$ disjoint from $B$ such that $\mu_\Phi(C)=\mu_\Phi(B)$. There is an involution $V\in[\mathcal R_\Phi]$ supported on $B\sqcup C$ such that $V(B)=C$. For all $n\in\N$ we let $$B_n=\{x\in B: d_\Phi(x,V(x))<n\text{ and }d_\Phi(U(x),VU(x))<n\}$$
Then $(B_n)$ is an increasing sequence of $U$-invariant Borel subsets of $B$ such that $\bigcup_{n\in\N} B_n=B$. Let $A_n$ be the $\Phi$-closure of $B_n$. Observe that $\bigcup_{n\in\N} A_n=X$, so the proof will be finished once the following claim is proven.

\begin{claim} For all $n\in\N$, the group $G$ contains $\overline{D([\Phi])}_{A_n}$.
\end{claim}

We now prove the claim. Let $n\in\N$; the sets and functions that are to be defined depend on $n$ but we will omit the subscript $n$ for readability. Let $F$ be a fundamental domain of the restriction of $U$ to its support, let $F_n$ be its intersection with $B_n$, which is a fundamental domain for the restriction of $U$ to $B_n$ .

Let $f=\mu_\Phi(F_n)=\mu_\Phi(B_n)/2$. We know that $\overline{D([\Phi])}_{A_n}$ is topologically generated by involutions, so it suffices to show that any involution in $[\Phi]_1$ supported in $A_n$ belongs to $G$. 

So let $W$ be an involution supported in $A_n$. Let $D$ be a Borel fundamental domain for the restriction of $W$ to its support: then $\supp W=D\sqcup W(D)$. By Corollary \ref{cor: cut a set in small subsets} we have a partition $(D_{m})_{m\in\N}$ of  $D$ such that for all $m\in\N$,
$\mu_\Phi(D_m)\leq \frac f 6$. For all $m\in\N$ we let $W_m=W_{D_m\sqcup W(D_m)}$. By the dominated convergence theorem,
$$W=\lim_{k\to+\infty}\prod_{m=0}^k W_m$$
so we only need to show that every $W_m$ belongs to $G$. 
\begin{center}
\includegraphics[scale=1.5]{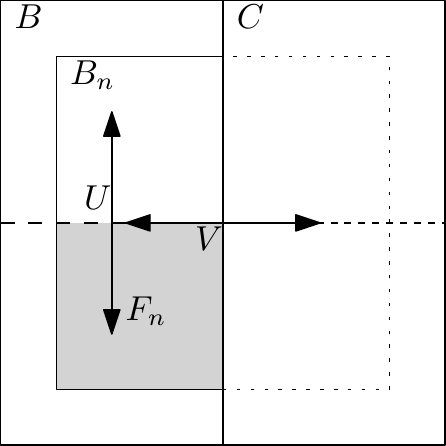}
\end{center}
Let us fix $m\in\N$, and for notational simplicity let $\Gamma=\la U,V\ra\simeq (\Z/2\Z\times \Z/2\Z)\rtimes\Z/2\Z$. Since $\mu_\Phi(D_m)\leq \frac {\mu_\Phi(F_n)}{10}$ we can find a subset $E\subseteq F_n$ such that
$\mu_\Phi(E)=\mu_\Phi(D_m)/2$ and  $E$ is disjoint from $\Gamma(D_m\sqcup W(D_m))$. Then $\Gamma(E)$ is disjoint from $D_m\sqcup W(D_m)=\supp W_m$.

Consider the $\Gamma$-invariant set $\tilde E=(E\sqcup U(E))\sqcup V(E\sqcup U(E))$ and the involution $\tilde V=V_{\tilde E}$. An easy calculation yields $U\tilde VU\tilde V=U_{\tilde E}$ so that $U_{\tilde E}\in G$. Now the support of $U_{\tilde E}$ has $\Phi$-conditional measure 
$$\mu_\Phi(\tilde E)=4\mu_\Phi(E)=2\mu_\Phi(D_m)=\mu_\Phi(\supp W_m).$$
Since moreover  $\tilde E$ and $\supp W_m$ are disjoint, we infer by Lemma \ref{lem: conjugate involutions} that $W_m$ belongs to the closure of the conjugacy class of of $U_{\tilde E}$. Since we have already established that $U_{\tilde E}\in G$, we can conclude that $W_m\in G$. As already observed, this ends the proof of the claim and therefore of the theorem.
\end{proof}

\begin{prop}\label{prop: closed normal gen by T}Let $\Phi$ be an aperiodic graphing,  let $T\in[\Phi]_1$ and let $A$ be the $\Phi$-closure of the support of $T$. Then the closure of the subgroup generated by conjugates of $T$ by elements of $\overline{D([\Phi]_1)}$ contains $$\overline{D([\Phi]_1)}_A.$$
\end{prop}
\begin{proof}
Again by Lemma \ref{lem: restriction of derived group} we only need to show that whenever $T$ has a support whose $\Phi$-closure is equal to $X$, the closure of the subgroup generated by conjugates of $T$ by involutions contains $\overline{D([\Phi]_1)}$.

Let $G$ be the closure of the subgroup generated by conjugates of $T$ by elements of $\overline{D([\Phi]_1)}$. We start by applying Proposition \ref{prop: small rohlin tower} and find a partition $(A_1,A_2,B_1,B_2,B_3)$ of $\supp T$ such that $A_2=T(A_1)$, $B_2=T(B_1)$ and $B_3=T^2(B_1)$. Let $C=A_1\sqcup b_1$, then by construction $C\cap T(C)=\emptyset$ and $C$ intersects every non-trivial $T$-orbit. Since the $\Phi$-closure of $\supp T$ is equal to $X$ this implies $C$ intersects every $\Phi$-orbit.

Let $G$ be the closure of the subgroup generated by conjugates of $T$ by elements of $\overline{D([\Phi]_1)}$. Let $U$ be an involution in $[\Phi]_1$ whose support is equal to $C$. For all $n\in\N$, let $C_n=\{x\in X, d_\Phi(x,U(x))<n\}$. Then $\bigcup_{n\in\N} C_n=C$ and for all $n\in\N$ we have $U_{C_n}\in[\Phi]_1$. Now since $U(C)$ is disjoint from $C$ we also have that $U(C_n)$ and $C_n$ are disjoint so that $[T,U_{C_n}]$ is an involution belonging to $G$ whose support is $C_n\sqcup T(C_n)$. 

Applying Lemma \ref{lem: closed normal gen by invol} we deduce that $G$ contains $\overline{D([\Phi]_1)}_{A_n}$ where $A_n$ is the $\Phi$-closure of $C_n$. Since $\bigcup_{n\in\N}C_n=C$ and $C$ intersects every $\Phi$-orbit, we have that $\bigcup_{n\in\N} A_n=X$. Moreover the sequence $(A_n)$ is increasing, so a direct application of Proposition \ref{prop: weak order continuity}  yields that $\bigcup_{n\in\N} \overline{D([\Phi]_1)}_{A_n}$ is dense in $\overline{D([\Phi]_1)}$ so that $G$ contains $\overline{D([\Phi]_1)}$ as wanted. 
\end{proof}

We can now easily prove our main results for this section. 

\begin{thm}\label{thm: closed normal subgroups}Let $\Phi$ be an aperiodic graphing. Then for every closed normal subgroup $N$ of $\overline{D([\Phi]_1)}$ there is a (unique up to measure zero) $\Phi$-invariant Borel set $A$ such that
$$N=\overline{D([\Phi]_1)}_A.$$
\end{thm}

\begin{proof}
Let $t=\sup \{\mu(A): A\text{ is }\Phi\text{-invariant and }\overline{D([\Phi]_1)}_A\subseteq N\}$, and let $A_n$ be a sequence of subsets of $X$ such that $\mu(A_n)\to t$ and $\overline{D([\Phi]_1)}_{A_n}\subseteq N$. Note that whenever $A', A''$ are two $\Phi$-invariant subsets of $X$ then the group generated by $\overline{D([\Phi]_1)}_{A'}\cup\overline{D([\Phi]_1)}_{A''}$ is equal to 
$\overline{D([\Phi]_1)}_{A'\cup A''}$. So we may assume the sequence $(A_n)$ is increasing.

Let $A=\bigcup_{n\in\N} A_n$, then $\mu(A)=t$. By Proposition \ref{prop: weak order continuity} since $N$ is closed we actually have that $N$ contains $\overline{D([\Phi]_1)}_A$. 
Conversely, let $T\in N$ and let $B$ be the $\Phi$-closure of the support of $T$. By Proposition \ref{prop: closed normal gen by T} the group $N$ contains $\overline{D([\Phi]_1)}_B$. So $N$ contains $\overline{D([\Phi]_1)}_{A\cup B}$ hence $\mu(A\cup B)\leq t$. But $\mu(A)=t$ so $B\subseteq A$ up to measure zero, hence $T\in \overline{D([\Phi]_1}_A$. So $\overline{D([\Phi]_1)}_A$ contains $N$ and we conclude that $N=\overline{D([\Phi]_1)}_A$.

Finally, the uniqueness of $A$ up to measure zero follows from the fact that given any $B\subseteq X$ of positive measure, there is a non trivial $T\in \overline{D([\Phi]_1)}$ supported in $B$ (see Proposition \ref{prop: chara many involutions}). 
\end{proof}

\begin{cor}Let $\Phi$ be an aperiodic graphing. Then $\Phi$ is ergodic if and only if $\overline{D([\Phi]_1)}$ is topologically simple.\qed
\end{cor}

We will see that $\overline{D([\Phi]_1)}$ and $D([\Phi]_1)$ are never simple when $\Phi=\{T\}$ and $T\in\Aut(X,\mu)$ is aperiodic (see Theorem \ref{thm:non simplicity}).

\subsection{Link with topological full groups}

In this section we connect $\LL^1$ full groups to the well-studied topological full groups. Since we are dealing with $\LL^1$ full groups associated to graphings, we need to introduce their natural topological counterparts. 

Let $2^\N$ be the Cantor space and let $\nu$ be a Borel probability measure on it. A \textbf{continuous graphing} on the measured Cantor space $(2^\N,\nu)$ is a graphing $\Phi$ such that for all $\varphi\in\Phi$, the sets $\dom \varphi$, $\rng\varphi$ are clopen and $\varphi: \dom\varphi\to \rng\varphi$ is a homeomorphism.

\begin{df}Let $\Phi$ be a continuous graphing on $(2^\N,\nu)$. The \textbf{topological full group} of $\Phi$, denoted by $[\Phi]_c$, is the group of all homeomorphisms $T$ of $2^\N$ such that for all $x\in 2^\N$ there is a neighborhood $U$ of $x$ and a $\Phi$-word $w$ such that for all $y\in U$, $T(x)=w(x)$. 
\end{df}

It follows from compactness that one can think equivalently of the topological full group of a continuous graphing $\Phi$ as the group of all maps obtained by cutting and pasting finitely many $\Phi$-words along a partition into clopen sets. In particular,  $[\Phi]_c$ is always a countable group. The most interesting case is when $\Phi$ is actually a countable group acting by homeomorphisms on a compact set (or a generating set for such a countable group): we then recover the usual definition of the topological full group. 

We denote by $[\Phi]_{c-inv}$ denote the group generated by involutions in a topological full group $[\Phi]_c$. The following result is an $\LL^1$ full group version of a result of Miller \cite{millerdensity} for full groups, also proved by Kittrell-Tsankov \cite[Prop. 4.1]{MR2599891} and Medynets \cite{MR2291714} in the case of a single homeomorphism.

\begin{thm}\label{thm: density gp gen by invol in topo ful group}
Let $\nu$ be a Borel probability measure on $2^\N$, and let $\Phi$ be a continuous graphing on $(2^\N,\mu)$. Then the group $[\Phi]_{c-inv}$ is dense in $\overline{D([\Phi]_1)}$. 
\end{thm}
\begin{proof}
Let us recall that if $\varphi\in\Phi$ and $A$ is a Borel subset of $\dom\varphi$  such that $A\cap \varphi(A)=\emptyset$, we denote by $I_{\varphi,A}$ the involution in $[\Phi]_1$ defined by: for all $x\in X$,
$$I_{\varphi,A}(x)=
\left\{\begin{array}{cl}\varphi(x) & \text{if }x\in A \\\varphi\inv(x) & \text{if }x\in \varphi(A) \\x & \text{otherwise.}\end{array}\right.$$
By Theorem \ref{thm: topo gen by involutions induced by graphing}, the group $\overline{D([\Phi]_1)}$ is topologically generated by such involutions, so we only need to check that, given $\varphi\in\Phi$ and a Borel subset $A$ of $\dom\varphi$  such that $A\cap \varphi(A)=\emptyset$, the involution $I_{\varphi,A}$ can be approximated by elements of $[\Phi]_{c-inv}$. 

So let $\varphi\in\Phi$ and let $A$ be a Borel subset of $\dom\varphi$  such that $A\cap \varphi(A)=\emptyset$. Since $\nu$ is regular, there is a sequence $(A_n)$ of clopen subsets of $2^\N$ such that $\nu(A_n\bigtriangleup A)\to 0$. 

Consider the set $A'_n:=A_n\setminus \varphi(A_n)$. Note that since $\varphi$ is a partial homeomorphism, $A'_n$ is still clopen, and by definition $A'_n$ is disjoint from $\varphi(A'_n)$ so the involution $I_{\varphi,A'_n}$ belongs to $[\Phi]_{c-inv}$.  We now have
\begin{align*}\tilde d_\Phi(I_{\varphi,A},I_{\varphi,A'_n)})=&\int_{A\setminus A'_n}d_\Phi(\varphi(x),x)+\int_{A'_n\setminus A}d_\Phi(x,\varphi(x))\\ &+\int_{\varphi(A\setminus A'_n)}d_\Phi(\varphi\inv(x),x)+\int_{\varphi(A'_n\setminus A)}d_\Phi(x,\varphi\inv(x))\\
=&2\mu(A'_n\bigtriangleup A).\end{align*}
Moreover since $\nu(A_n\bigtriangleup A)\to 0$, since $\varphi$ preserves the measure and since $A$ is disjoint from $\varphi(A)$, a straightforward computation yields $\mu( A'_n\bigtriangleup A)\to 0$.
We conclude that $\tilde d_\Phi(I_{\varphi,A},I_{\varphi,A'_n)})\to 0$ as desired. 
\end{proof}

Let us now show that the similar statement that $[\Phi]_c$ is dense in $[\Phi]_1$ is not true in general.

\begin{exemple}\label{example: full topological not dense in L1}Consider the odometer $T_0$ on $2^\N$ equipped with the unique $T_0$-invariant probability measure. For every word $s$ with letters in $\{0,1\}$, we let $N_s\subseteq 2^\N$ be the set of infinite words which start by $s$. For every $n\in\N$, let $\varphi_n$ be the restriction of $T_0$ to $N_{1^n0}$, and then let $\Phi=\{\varphi_n: n\in\N\}$. It is easily checked that the topological full group $[\Phi]_c$ is the group $\mathfrak S_{2^\infty}$ of dyadic permutations. In particular $[\Phi]_c$ is locally finite, hence contained in $\overline{D([\Phi]_1)}$. We conclude from Theorem \ref{thm: density gp gen by invol in topo ful group} that $\overline{[\Phi]_c}=\overline{D([\Phi]_1)}$. 

However the $\LL^1$ full group of $\Phi$ is equal to the $\LL^1$ full group of $T_0$ since up to measure zero they induce the same graph. But $T_0$ has index $1$ (see Definition \ref{df:index}) hence cannot belong to $\overline{D([T_0]_1)}=\overline{D([\Phi]_1)}$, and we conclude that $[\Phi]_c$ is not dense in $[\Phi]_1$. 
\end{exemple}

It is well-known that every countable group $\Gamma$ of measure-preserving transformation of $(X,\mu)$ is up to measure zero conjugate by a measure-preserving isomorphism to a group of homeomorphism of the Cantor space preserving a Borel probability measure $\nu$ (see for instance \cite[Thm. 2.15]{MR1958753}). Here we check that the same can be done for graphings so as to be able to apply our density result above in the right level of generality.

Recall that the group of homeomorphisms of the Cantor space is canonically isomorphic to the group of automorphisms of the Boolean algebra of clopen subsets of the Cantor space. Furthermore, the Boolean algebra of clopen subsets of the Cantor space is the unique countable atomless Boolean algebra (see for instance \cite[Chap. 16]{MR2466574}). We can now state and prove that any graphing is conjugate to a continuous one up to measure zero. 

\begin{prop}Let $(X,\mu)$ be a standard Borel space, and let $\Phi$ be a graphing on $X$. Then there exists a Borel probability measure $\nu$ on $2^\N$, a continuous graphing $\Psi$ on $(2^\N,\nu)$, a $\Phi$-invariant full measure Borel subset $X_0$ of $X$ and Borel measure-preserving injection $T: (X_0,\mu)\to (2^\N,\nu)$ and a continuous graphing $\Psi$ on $2^\N$ such that  $$\Psi_{\restriction T(X_0)}=T\Phi_{\restriction X_0} T\inv.$$\end{prop}
\begin{proof}

Let $\mathcal C$ be a countable separating family of  Borel sets containing $X$. Let $\mathcal D$ be  set of finite Boolean combinations of sets of the form $w(C)$ or $w\inv(C)$ for $w$ a $\Phi$-word and $C\in\mathcal C$. Then $\mathcal D$ is countable separating, contains the domain and range of every $\varphi\in\Phi$, and is $\Phi$-invariant in the sense that for every $\varphi\in\Phi$ and $D\in\mathcal D$, $\varphi^{±1}(D)\in \mathcal D$. 

Each atom of $\mathcal D$ is a singleton since $\mathcal D$ contains a separating family. Let $A$ be the union of the atoms of $D$, then $A$ is countable and hence has measure zero. Let $X_0:=X\setminus A$, and consider the Boolean algebra $\mathcal D_{X_0}:=\{D\cap X_0: D\in\mathcal D\}.$ induced by $\mathcal D$ on $X_0$.

By Stone duality, $\mathcal D_{X_0}$ is isomorphic to the  algebra of clopen subsets of its Stone space $Y$. 
Since $\mathcal D_{X_0}$ is a countable atomless algebra its Stone space $Y$ is homeomorphic to $2^\N$ so we may as well assume that $Y=2^\N$. 

By construction every $\varphi\in\Phi$ induces a Boolean algebra isomorphism between the Boolean algebra induced by $\mathcal D_{X_0}$ on $\dom\varphi$ and the Boolean algebra induced by $\mathcal D_{X_0}$ on $\rng\varphi$, so every $\varphi\in \Phi$ induces a homeomorphism $\tilde \varphi$ between the corresponding clopen subsets of $Y$. Let $\Psi=\{\tilde \varphi: \varphi\in\Phi\}$, we will check that such a continuous graphing works.

We define a map $T:X_0\to Y$ by associating to every $x\in X_0$ the unique element $T(x)$ of the Stone space of $\mathcal D_{X_0}$ defined by $x$. Since $\mathcal D_{X_0}$ contains a separating family of $X_0$, the map $T$ is injective. Moreover it conjugates the restriction $\Phi_{\restriction X_0}$ and $\Psi_{\restriction T(X_0)}$ by the definition of $\Psi$. Finally the map $T$ is also Borel since the preimage of a clopen set belongs to $\mathcal D_{X_0}$, so by putting $\nu=T_*\mu$ we get the desired statement. 
\end{proof}

\section{$\LL^1$ full groups of $\Z$-actions}\label{sec: L1 full groups of Z actions}

We now focus on $\LL^1$ full groups of $\Z$-actions, for which we can say a lot more. The main feature is that for $\Z$-actions, $\LL^1$ full groups become invariants of \textit{flip-conjugacy}. To see this, we need the following fundamental result of Belinskaya. 

\begin{thm}[Belinskaya, \cite{belinskaya1968partitions}]\label{thm: belink L1OE implies flip conj}
Let $T$ and $T'$ be two ergodic transformations with the same orbits, and suppose furthermore that $T'\in[T]_1$. Then there exists $S\in[T]$ such that either 
$$T=ST'S\inv\text{ or }T\inv=ST'S\inv.$$
\end{thm}

With this in hand, we can now prove that $\LL^1$ full groups of ergodic $\Z$-actions are complete invariants of flip conjugacy. 

\begin{thm}\label{thm:l1fullgroup is complete invariant of flip conjugacy}Let $T$ and $T'$ be two measure-preserving ergodic transformation of a standard probability space $(X,\mu)$. Then the following are equivalent: 
\begin{enumerate}[(1)]
\item $T$ and $T'$ are flip-conjugate;
\item the groups $[T]_1$ and $[T']_1$ are abstractly isomorphic;
\item the groups $[T]_1$ and $[T']_1$ are topologically isomorphic. 
\end{enumerate}
\end{thm}
\begin{proof}
Clearly $(1)$ implies $(3)$ and $(3)$ implies $(2)$. Suppose that $(2)$ holds and let $\psi:[T]_1\to[T']_1$ be an abstract group isomorphism. By Corollary \ref{cor: Dye reconstruction for L1}, there is $S\in\Aut(X,\mu)$ such that for all $U\in [T]_1$, we have 
$$\psi(U)=SUS\inv.$$
In particular, we have that $STS\inv\in [T']_1$. Since $\psi$ is surjective, $STS\inv$ must have the same orbits as $T'$. So by Belinkskaya's Theorem \ref{thm: belink L1OE implies flip conj}, $STS\inv$ and $T'$ are flip-conjugate, so $T$ and $T'$ are flip-conjugate. So $(2)$ implies $(1)$, which ends the proof. 
\end{proof}

\subsection{The index map}

One of the specific features of $\LL^1$ full groups of aperiodic $\Z$-actions is the existence of a natural homomorphism into $\R$, given by integrating the corresponding cocycles. We will see later that in the ergodic case, it actually takes values into $\Z$. 

\begin{df}\label{df:index}Let $T\in\Aut(X,\mu)$ be ergodic. The \textbf{index} of $S\in[T]_1$ is defined by
$$I_T(S):=\int_Xc_S(x)d\mu(x),$$
where $c_S$ is the cocycle associated to $S\in[T]$. When the transformation $T$ we are considering is clear from the context, we will also simply write the index map as $I$. 
\end{df}

From the cocycle identity $c_{SS'}(x)=c_S'(S(x))+c_S(x)$ and the fact that elements of $[T]_1$ are measure preserving we deduce that the index map is a homomorphism. We will denote its kernel by $[T]_0$. 

Our first observation is that taking induced transformations does not change the index. The proof is very similar to that of Kac's formula, which should be no surprise since Kac's formula corresponds to the case $S=T$. 

\begin{prop}\label{prop:indexinduced}Let $S\in[T]_1$, and let $A\subseteq X$ intersect every non trivial $S$-orbit. Then we have the equality 
$$I(S_A)=I(S).$$
\end{prop}
\begin{proof}
For all $n\in\N$, let $A_n:=\{x\in X: T_A(x)=T^n(x)\}$. Then the $A_n$'s form a partition of $A$. For all $n\in\N$ and all $0\leq m<n$, we let $B_{n,m}:=T^m(A_n)$. The family $(B_{n,m})_{0\leq m<n}$ is a partition of the support of $S$ since $A$ intersects every non trivial $S$-orbit. We now compute 
\begin{align*}
I(S_A)=&\int_Ac_{S_A}(x)d\mu(x)\\
&=\sum_{n\in\N}\int_{A_n}c_{S_A}(x)d\mu(x)\\
&=\sum_{n\in\N}\int_{A_n}\sum_{m=0}^{n-1}c_S(S^m(x))d\mu(x),
\end{align*}
the last equality being a consequence of the cocycle identity. We deduce that 
\begin{align*}
I(S_A)&=\sum_{n\in\N}\sum_{m=0}^{n-1}\int_{A_n}c_S(S^m(x))d\mu(x)\\
&=\sum_{n\in\N}\int_{B_{n,m}}c_S(x)d\mu(x)\\
&=\int_{\supp S}c_S(x)d\mu(x)\\
I(S_A)&=I(S)\qedhere
\end{align*}
\end{proof}

The following corollary can also be obtained directly by noting that if $S\in[T]_1$, then on every finite $S$-orbit the average of $n_S$ must be equal to zero. 

\begin{cor}\label{cor: periodic has index zero}Let $T\in\Aut(X,\mu)$ be aperiodic. Every periodic element of $[T]_1$ has index zero.
\end{cor}
\begin{proof}Let $S\in[T]_1$ be a periodic element, let $A\subseteq X$ be a fundamental domain for $S$ (see the paragraph following Definition \ref{df: periodic}). Then $S_A=\mathrm{id}_X$ by definition so $I(S_A)=0$. By the previous proposition we conclude that $I(S)=I(S_A)=0$. 
\end{proof}

One can completely describe the ergodic elements of the $\LL^1$ full group of $T\in\Aut(X,\mu)$ which are conjugate to $T$ in terms of their index. This will be very useful to us in the next section. 

\begin{thm}[{Belinskaya, \cite[Thm. 3.8]{belinskaya1968partitions}}]\label{thm: conjugate and index}
Let $T\in\Aut(X,\mu)$ be an ergodic transformation, and let $T'\in[T]_1$. Then $T$ and $T'$ are flip-conjugate iff $T'$ is ergodic and $I(T')=±1$.
\end{thm}

\subsection{A characterization of supports of $n$-cycles}\label{sec: chara supports}

In full groups of aperiodic equivalence relations, every Borel set is the support of an involution. Here we will see that this is not true in general for $\LL^1$ full groups and give a complete characterization of sets which arise as supports of involutions, or more generally of $n$-cycles.

\begin{df}A measure-preserving transformation $T$ is an $n$\textbf{-cycle} if every $T$-orbit has cardinality either $1$ or $n$. 
\end{df}

Note that if $p$ is prime, a $p$-cycle is precisely an element of $\Aut(X,\mu)$ of order $p$. Our result is that $n$-cycles in $\LL^1$-full groups arise only in the most obvious manner: for instance there is an involution whose support is equal to $X$ if and only if there exists a Borel set $B$ such that $X=B\sqcup T(B)$.

\begin{thm}\label{thm: chara which sets are supports of n cycles}
Let $T$ be a measure-preserving ergodic transformation,  let $A$ be a Borel subset of $X$ and let $n\in\N$. Then $A$ is equal to the support of an $n$-cycle if and only if there is $B\subseteq A$ such that 
$$A=B\sqcup T_A(B)\cdots\sqcup T_A^{n-1}(B)$$
\end{thm}
\begin{rmq}We should point out however that every Borel set arises as the support of a periodic transformation all whose orbits have cardinality at most 3 (see Lemma \ref{lem: maximal partition of support}).
\end{rmq}
\begin{proof}[Proof of Theorem \ref{thm: chara which sets are supports of n cycles}]
The sufficiency of the condition is easy:  since $T_A\in[T]_1$ by Kac's formula (see Prop. \ref{prop: induced transfo}), if $A=B\sqcup T_A(B)\cdots\sqcup T_A^{n-1}(B)$ we can define an $n$-cycle $T_0\in[T]_1$ whose support equates $A$ by putting 
$$T_0(x)=\left\{\begin{array}{cl}T_A(x) & \text{if }x\in \bigsqcup_{i=0}^{n-2}T_A^i(B)  \\
T_A^{-n-1}(x) & \text{if  }x\in T_A^{n-1}(B) \\
x & \text{otherwise.}\end{array}\right.$$

For the converse, let us recall that if $G\leq \Aut(X,\mu)$ and $A$ is a Borel subset of $X$, we let $G_A$ denote the group of elements of $G$ supported in $A$. Observe that if $T\in\Aut(X,\mu)$ then for all $x,y\in A$ we have $d_T(x,y)\geq d_{T_A}(x,y)$ so that ${[T]_1}_A\leq [T_A]_1$ (we will see in Theorem  \ref{thm:non stability under l1} that  the inclusion is actually strict). This observation will be used implicitly in what follows.

Suppose that $A$ is the support of an $n$-cycle $T_0$, let $C$ be a fundamental domain for $T_0$, and consider the transformation $S=T_CT_0$ whose support  also equates $A$ and satisfies 
$$A=C\sqcup S(C)\sqcup\cdots\sqcup S^{n-1}(C).$$ By Proposition \ref{prop: building ergodic and aperiodic from periodic} the restriction of $S$ to $A$ is ergodic. Moreover note that $T_C=(T_A)_C$ so the the $T_A$-index of $T_C$ is $I_{T_A}(T_C)=1$ by Proposition \ref{prop:indexinduced}. Since $T_0$ is periodic, its $T_A$-index is equal to $0$ by Corollary \ref{cor: periodic has index zero} and we conclude that 
$$I_{T_A}(S)=1.$$
By applying Belinskaya's theorem \ref{thm: conjugate and index} to the restriction of $S$ to $A$ which is ergodic, we conclude that $T_A$ and $S$ are conjugate by a measure-preserving transformation $U$ supported in $A$. But recall that $A=C\sqcup S(C)\sqcup\cdots\sqcup S^{n-1}(C)$ so we have $A=U(C)\sqcup T_A(U(C))\cdots\sqcup T_A^{n-1}(U(C))$ hence $T_A$ satisfies the desired conclusion with $B:=U(C)$.
\end{proof}
\begin{exemple}Let $T\in\Aut(X,\mu)$ be a weakly mixing transformation. Then for all $n\in\N$, there is no $n$-cycle in $[T]_1$ whose support equates $X$. 
\end{exemple}

\begin{rmq}The existence of $B\subseteq A$ such that 
$A=B\sqcup T_A(B)\cdots\sqcup T_A^{n-1}(B)$
is a spectral property of $T_A$: it is equivalent to the associated unitary operator on $\LL^2(A)$ having $e^{2i\pi/n}$ as an eigenvalue. 
\end{rmq}

\subsection{The $\LL^1$ full group is generated by induced transformations together with periodic transformations}\label{sec:generated by induced and periodic}

We will now use the index map to find a nice generating set for $[T]_1$. Our ideas are in the direct continuation of Belinskaya's paper \cite{belinskaya1968partitions}. Let us start by analyzing the action of elements of $[T]_1$ on the $T$-orbits. 

A sequence $(n_{k})_{k\in\N}$ of integers is called \textbf{almost positive} if for all but finitely $k\in\N$, we have $n_k\geq 0$. It is called \textbf{almost negative} if $(-n_k)_{k\in\N}$ is almost positive, and a sequence which is either almost negative or almost positive is called \textbf{almost sign invariant}. It is an easy exercise to check that an injective sequence is almost positive iff it tends to $+\infty$, while it is almost negative iff it tends to $-\infty$. The following lemma thus implies that under the identification of a $T$-orbit with $\Z$ via $k\mapsto T^k(x)$, every forward $S$-orbit either tends to $+\infty$ or $-\infty$.

\begin{lem}[{\cite[Thm. 3.2]{belinskaya1968partitions}}] Let $T\in\Aut(X,\mu)$ be aperiodic and  $S\in [T]_1$. For almost every $x\in X$ such that the $S$-orbit of $x$ is infinite, the sequence $sign(c_{S^k}(x))$ is almost sign invariant.
\end{lem}
%
%
%
%

\begin{df}An element $S\in[T]_1$ such that for all $x\in \supp S$, the sequence $(c_{S^k}(x))_{k\in\N}$ is almost positive is called \textbf{almost positive}. If for  every $x\in \supp S$ the sequence $(c_{S^k}(x))_{k\in\N}$ is almost negative, we say that $S$ is \textbf{almost negative}. 
\end{df}

Note that every almost positive or almost negative element has to be aperiodic when restricted to its support. Also the inverse of an almost positive element is almost negative and vice-versa. We have a natural order on every $T$-orbit which allows for a more natural formulation of the previous definition.

\begin{df}
Let $T\in\Aut(X,\mu)$ be aperiodic. We define a partial order $\leq_T$ on $X$ by $x\leq_T y$ if $y=T^n(x)$ for some $n\geq 0$.
\end{df}

Note that $\leq_T$ is total when restricted to a $T$-orbit, and that such a restriction corresponds to the natural ordering $\leq$ on $\Z$ via $n\mapsto T^n(x)$. 

We may now reformulate the previous definition as: $S\in[T]_1$ is almost positive if for almost every $x\in\supp S$, $S^k(x)\geq_Tx$ for all but finitely many $k\in\N$.

\begin{prop}\label{prop:decomposition into almost +- and periodic}Let $T\in\Aut(X,\mu)$ be aperiodic. Every element $s\in [T]_1$ can be written as a product
$S=S_pS_+S_-$, where $S_p$ is periodic, $S_+$ is almost positive, $S_-$ is almost negative and these three elements have disjoint supports. 
\end{prop}
\begin{proof}
Let $S\in[T]_1$. The set of $x\in X$ such that the sequence $(c_{S^k}(x))_{k\in\N}$ is almost positive (respectively almost negative) is $S$-invariant. Call this set $A_+$ (respectively $A_-$). Finally let $A_p$ denote the set of $x\in X$ whose $S$-orbit is finite. Then $(A_-,A_+, A_p)$ is a partition of $X$ by the previous lemma, and since they are $S$-invariant we deduce that $S=S_{A_-}S_{A_+}S_{A_p}$. Clearly $S_p:=S_{A_p}$, $S_+:=S_{A_+}$ and $S_-:=S_{A_-}$ are as wanted.
\end{proof}

\begin{df}Let $T\in\Aut(X,\mu)$ be aperiodic. If $S\in[T]_1$, we say that $S$ is \textbf{positive} if $S(x)\geq_T x$ for all $x\in X$. 
\end{df}

The proof of the following proposition uses the same ideas as in \cite[Lem. 3.5]{belinskaya1968partitions}. 

\begin{prop}\label{prop: almost positive is positive up to a periodic}Let $T\in\Aut(X,\mu)$ be aperiodic, and let $S\in[T]_1$ be almost positive.  Then there exists a positive element $S'\in[T]_1$ whose support is a subset of $\supp S$ such that $S{S'}\inv$ and ${S'}\inv S$ are periodic. 
\end{prop}
\begin{proof}
Consider the set $A:=\{x\in X: S^k(x)>_Tx\text{ for all }k>0\}$. Note that $T_A(x)$ may be defined as the first element greater than $x$ which belongs to $A$, so by the definition of $A$ we have $T_A(x)\leq_T S_A(x)$ for all $x\in A$. 

Let us show that $A$ intersects every non trivial $S$-orbit. Let $x\in \supp S$. Since $S$ is almost positive, the sequence $(n_{S^k}(x))_{k\in\N}$ is almost positive. Let $k$ be the last natural integer such that $n_{S^k(x)}\leq 0$, then by definition $S^k(x)\in A$ as wanted.

Now for all $x\in A$ we have $S_A(x)\geq_TT_A(x)\geq_T x$, in particular $S_A$ is positive. Since $A$ intersects every $S$-orbit, we moreover have that $SS_A\inv$ and $S_A\inv S$ are periodic, so $S':=S_A$ is as desired. 
\end{proof}

\begin{lem}
Let $T\in\Aut(X,\mu)$ be aperiodic, let $S\in[T]_1$ be positive and let $A=\supp S$. Then $ST_A\inv$ is positive.
\end{lem}
\begin{proof}
Since $S$ is positive, we have $A= \{x\in X: S(x)>_Tx\}$. But then by the definition of $T_A$,for all $x\in A$ we have $S(x)\geq_T T_A(x)$. Since $A$ is $T_A$-invariant, we deduce that for all $x\in A$, $ST_A\inv(x)\geq_T x$. But the support of $ST_A\inv$ is contained in $A$ so the previous inequality actually holds for all $x\in X$, and we conclude that $ST_A\inv$ is positive. 
\end{proof}

\begin{prop}\label{prop: positive is product of induced}Let $T\in\Aut(X,\mu)$ be ergodic. Then every positive element is the product of finitely many  transformations which are induced by $T$: for all $S\in[T]_1^+$ there exists $k\in\N$ and $A_1,...,A_k\subseteq X$ such that $$S=T_{A_1}\cdots T_{A_k}.$$
In particular, the index of a positive element is a (positive) integer. 
\end{prop}
\begin{proof}
Note that the index of a positive element has to belong to $[0,+\infty[$, and that the index of a positive element $S$ is null if and only if $S=\mathrm{id}_X$. We prove the result by induction on $\lfloor I(S)\rfloor$.

First, if $\lfloor I(S)\rfloor=0$, suppose towards a contradiction that $S$ is non-trivial. let $A=\supp S$, by the previous lemma $ST_A\inv$ is positive. Since $T$ is ergodic, $A$ intersects almost every $T$-orbit. Proposition \ref{prop:indexinduced} thus yields that $I(T_A)=I(T)=1$. So $I(ST_A\inv)=I(S)-1<0$, but $ST_A\inv$ is positive, a contradiction.

The inductive step is basically the same: if $\lfloor I(S)\rfloor=n+1$ for some $n\in\N$, the support $A$ of $S$ is nontrivial so that $S=T_AS'$ where $\lfloor I(S')\rfloor=n$. 
\end{proof}

\begin{thm}\label{thm: decomposition of elements}Let $T\in\Aut(X,\mu)$ be ergodic, then $[T]_1$ is generated by the set $Per([T]_1)\cup Ind([T]_1)$, where $Per([T]_1)$ is the set of periodic elements in $[T]_1$ and $Ind([T]_1)$ is its set of transformations induced by $T$.  
\end{thm}
\begin{proof}
By Proposition \ref{prop:decomposition into almost +- and periodic} the group $[T]_1$ is generated by almost positive elements together with periodic elements. Now every almost positive element is the product of a positive element by a periodic element (Proposition \ref{prop: almost positive is positive up to a periodic}) and every positive element is the product of finitely many induced transformations (Proposition \ref{prop: positive is product of induced}). 
\end{proof}

\begin{cor}\label{cor: generated by T along with periodic}Let $T\in\Aut(X,\mu)$ be ergodic, then $[T]_1$ is generated by the set $Per([T]_1)\cup \{T\}$, where $Per([T]_1)$ is the set of periodic elements in $[T]_1$ .
\end{cor}
\begin{proof}
If $A$ is a non-null subset of $X$, then by ergodicity $A$ intersects almost every $T$-orbit so that $T_A\inv T$ is periodic. In other words every induced transformation belongs to the group generated by periodic elements together with $T$, and since by the previous theorem periodic elements together with induced transformations generate $[T]_1$  we are done. 
\end{proof}

\begin{cor}\label{cor:index take values in Z}Let $T\in\Aut(X,\mu)$ be ergodic. Then $\overline{D([T]_1)}$ is equal to the following three groups:
\begin{itemize}
\item the group generated by periodic elements,
\item the group topologically generated by involutions and
\item the kernel of the index map.
\end{itemize}

Moreover we have the following short exact sequence 
$$1\to \overline{D([T]_1)}\to [T]_1 \xrightarrow I\Z\to 1.$$
\end{cor}
\begin{proof}
Denote by $G_{per}$ the group generated by periodic elements in $[\Phi]_1$, and by $\overline{G_{inv}}$ the group topologically generated by involutions. 
Note that both are normal subgroups of $[T]_1$. By Lemma \ref{lem: periodic in derived group}, we have $G_{per}\leq \overline{G_{inv}}\leq\overline{D([T]_1)}$. Since the index map is continuous and takes values in an abelian group, we also have $\overline{D([T]_1)}\leq \ker(I)$. 

Let $S\in[T]_1$. We may write $S$ as a product of periodic elements and powers of $T$ by the previous corollary. So modulo $G_{per}$, we have $S=T^n$ for some $n\in\Z$. So $I(S)=n$ and we conclude that the index map takes values into $\Z$. Moreover if $S\in\ker(I)$  we infer that  $n=0$ so that $S\in G_{per}$. So $G_{per}=\ker(I)$ and since we have already established that $G_{per}\leq \overline{G_{inv}}\leq\overline{D([T]_1)}\leq \ker (I)$, we conclude that all these four normal subgroups are equal.
\end{proof}

\begin{rmq}Note that for topological full groups, the kernel of the index map can be larger than the derived group due to the fact that some involutions may not belong to the derived group. Matui has obtained a complete description of the quotient of the kernel of the index map by the derived group, see \cite[Thm. 4.8]{MR2205435}. 
\end{rmq}

\subsection{Escape time and non-simplicity  results}\label{sec:nonsimple}

In this section, we show that the closure of the derived group of $[T]_1$ is never simple. To this end, we will use the non-integrability of certain \textit{escape times}. 

Let us first recall that given a standard probability space $(X,\mu)$, its \textbf{measure algebra} is denoted by $\MAlg(X,\mu)$ and consists of all Borel subsets of $X$, two such sets being identified if their symmetric difference has measure zero. 

It is naturally endowed with a metric $d_\mu$ defined by $d_\mu(A,B)=\mu(A\bigtriangleup B)$ which is complete and separable (see \cite[40.A]{MR0033869}).

\begin{df}Let $T\in\Aut(X,\mu)$ and $A\in\MAlg(X,\mu)$. The $T$-\textbf{escape time} of $A$ is the map $\tau_A: A\to \N$ defined by\begin{itemize}
\item $\tau_A(x)=+\infty$ if there is no $n\in\N$ such that $T^n(x)\not\in A$,
\item  else $\tau_A(x)$ is the the first integer $n\in\N$ such that $T^{n}(x)\not\in A$ or $T^{-n}(x)\not\in A$.
\end{itemize}
\end{df}

By Kac's theorem, for any Borel set $A\subseteq X$ the $T$-return time to $A$ is integrable and its integral is equal to $1$. The situation for escape times is however very different. 

Indeed, any non-null $T$-invariant set fails to have an integrable escape time so a non ergodic transformation admits non-trivial sets with non-integrable escape time. This is actually true of any measure-preserving transformation in a strong sense as the following result shows.

\begin{thm}\label{thm:escape not integrable} Let $T$ be a measure-preserving transformation. Then there is a dense $G_\delta$ subset of $A\in\MAlg(X,\mu)$ whose $T$-escape time is not integrable.
\end{thm}

\begin{proof}For $n\in\N$, let $\mathcal F_n$ be the set of $A\in\MAlg(X,\mu)$ whose escape time is either non integrable or has $\LL^1$ norm stricly bigger than $n$. We have to show that each $\mathcal F_n$ is open dense. To see that $\mathcal F_n$ is open, first note that for a fixed $k\in\N$, if we let $A_k:=\{x\in A: n_A(x)=k\}$ then
$$A_{k+1}=\left(A\setminus \bigcup_{l=1}^kA_l\right)\cap\left( T^{-(k+1)}(A)\cup T^{k+1}(A)\right)$$
so that $A_k$ depends continuously on $A$. Now take $A\in\mathcal F_n$, by definition we may find $N\in\N$ big enough so that $\sum_{k=1}^Nk\mu(A_k)>n$. Such an inequality will remain true for $A'$ sufficiently close to $A$ by continuity, so that $\mathcal F_n$ is indeed open.

Next, to check the density, we will be done if we can build a sequence $(A_m)$ such that $\norm{\tau_{A_m}}_1\to +\infty$ but $\mu(A_m)\to 0$. Indeed, then $A\cup A_m$ converges to $A$ and for any $n\in\N$, it belongs to $\mathcal F_n$ for $m$ large enough. 

If the ergodic decomposition of $T$ has a diffuse part, there is a sequence of $T$-invariant sets $(A_m)$ such that $\mu(A_m)\to 0$ and $\norm{\tau_{A_m}}_1=+\infty$ by $T$-invariance.

If the ergodic decomposition of $T$ does not have a diffuse part, then $T$ has to be aperiodic. So we apply Rohlin's lemma to $T$ for $N=6^m$, $\epsilon=\frac 12$ to get $B_m\subseteq X$ all whose first $6^m$ translates are disjoint such that $\frac 1{2\cdot 6^m}\leq \mu(B_n)\leq \frac1{6^n}$. Let $$A_n:=\bigsqcup_{i=0}^{4^m-1}T^i(B_m)$$be the union of the first $4^m$ translates of the base $B_m$ of the Rohlin tower. We now compute the integral of the escape time:
\begin{align*} \int_{A_m}\tau_{A_m}&=2\mu(B_m)\sum_{i=0}^{4^m/2-1} (i+1)\\
&=2\mu(B_m)\cdot \frac{4^m}2 \frac{\frac{4^m}2+1}2\\
&\geq \frac 1{6^m}\cdot \frac{16^m}{8}
\end{align*}
which tends to $+\infty$ while $\mu(A_m)\leq (\frac46)^m$ tends to zero as $m$ tends to $+\infty$.
\end{proof}

\begin{rmq} For any $T\in\Aut(X,\mu)$ of rank one, it is easily checked that the set of $A\in\MAlg(X,\mu)$ whose $T$-escape time is bounded is also dense, but I do not know wether this is true for any $T\in\Aut(X,\mu)$. 
\end{rmq}

The following corollary of Theorem \ref{thm:escape not integrable} is in sharp contrast with the situation for full groups of ergodic measure-preserving equivalence relations, which act transitively on elements of $\MAlg(X,\mu)$ of the same measure. 

\begin{cor}\label{cor: cannot send to disjoint}
Let $T\in\Aut(X,\mu)$ be ergodic. There is a dense $G_\delta$ subset of $A\in \MAlg(X,\mu)$ such that for every $B\in\MAlg(X,\mu)$ disjoint of $A$ of the same measure, no element of $[T]_1$ maps $A$ to $B$. \end{cor}
\begin{proof}
If $T$ maps $A$ to a set disjoint from $A$, then by definition the $\LL^1$-norm of $T$ is an upper bound of the integral of the escape time of $A$. By the previous theorem, this cannot happen for a dense $G_\delta$ of $A\in\MAlg(X,\mu)$. 
\end{proof}

We will now refine the preceding result to establish that for all $T\in[\mathcal R]$ aperiodic, neither the derived group of $[T]_1$ nor its closure are simple. We will need the following lemma, which is essentially contained in \cite{MR0486410}.

\begin{lem} \label{lem: maximal partition of support}Let $T\in\Aut(X,\mu)$ be aperiodic. For every $A\in\MAlg(X,\mu)$, there exists a periodic $U\in[T]_1$  whose support is equal to $A$ and all whose orbits have cardinality at most $3$. 
\end{lem}
\begin{proof}
Consider the induced transformation $T_A$, which belongs to $[T]_1$ by Proposition \ref{prop: induced transfo}. Since $T$ is aperiodic, $\supp T_A=A$. By a maximality argument, there is a partition $(A_1,A_2,B_1,B_2,B_3)$ of $A$ such that $T_A(A_1)=A_2$, $T_A(B_1)=B_2$ and $T_A(B_2=B_3)$ (see \cite[Lem. 7.2]{MR0435350}).
We then have a naturally defined periodic element $U\in[T]_1$ given by
$$U(x)=\left\{\begin{array}{cl}T_A(x) & \text{if }x\in A_1\sqcup B_1\sqcup B_2 \\T_A\inv(x) & \text{if }x\in A_2 \\T_A^{-2}(x) & \text{if }x\in B_3 \\x & \text{else.}\end{array}\right.$$
It is easily checked that $U$ is as desired.
\end{proof}

\begin{lem}\label{lem: commutator with support A}Let $T\in\Aut(X,\mu)$ be aperiodic. For every Borel $A\subseteq X$, there exists $U\in D([T]_1)$ whose support is equal to $A$ and all whose orbits have cardinality at most $3$. 
\end{lem}
\begin{proof}
Consider the transformation $U\in[T]_1$ provided by the previous lemma. Observe that $U=U_2U_3$ where $U_2$ is the transformation induced by $U$ on the set of points whose $U$-orbit has cardinality $2$, and $U_3$ is the transformation induced by $U$ on the set of points whose $U$-orbit has cardinality $3$. 

The transformation $U_3$ is a commutator by the equality $(1\;2\;3)=(1\; 2)(1\; 3)(1\; 2)(1\;3)$ in the symmetric group over $\{1,2,3\}$. 

    Let $C$ be a fundamental domain of the restriction of $U_2$ to its support. Applying the above lemma again, we find $V\in[T]_1$ all whose orbits have cardinality at most $3$ such that $\supp V=C$. We then let $\tilde U_2=U_2VU_2V\inv\in D([T]_1)$ which has the same support as $U_2$. Now $U'=\tilde U_2U_3$ is as desired. 
\end{proof}

If every $A\in\MAlg(X,\mu)$ were the support of an \textit{involution} belonging to $[T]_1$, one could take $U$ from the previous lemma to be an involution. However $X$ itself may fail to be the support of an involution (see Theorem \ref{thm: chara which sets are supports of n cycles} and the example thereafter).

\begin{thm}\label{thm:non simplicity}
Let $T\in\Aut(X,\mu)$ be ergodic. Then neither $D([T]_1)$ nor $\overline{D([T]_1)}$ are simple.
\end{thm}

The previous theorem follows trivially from the following statement which we will prove using Theorem \ref{thm:escape not integrable}.

\begin{lem}
Let $T\in\Aut(X,\mu)$ be ergodic. Then there exists $S_1,S_2\in D([T]_1)$ such that $S_1$ cannot be written as a product of conjugates of $S_2^{\pm 1}$ by elements of $[T]_1$.  
\end{lem}
\begin{proof}
Take $A\in\MAlg(X,\mu)$ whose escape time is not integrable, which exists by Theorem \ref{thm:escape not integrable}. 

By Lemma \ref{lem: commutator with support A}, we may fix $S_1\in D([T_1])$ whose support is equal to $A$ and $S_2\in D([T]_1)$ whose support $B$ is disjoint from $A$. Now assume $S_1$ is a product of conjugates of $S_2^{±1}$ by $T_1,...,T_n\in [T]_1$. We must have $A\subseteq\bigcup_{i=1}^nT_i(B)$, but then the integral of the escape time of $A$ must be less than 
$$\sum_{i=1}^nd_T(T_i\inv,\mathrm{id}_X)<+\infty,$$
a contradiction. 
\end{proof}
\begin{rmq}The two above results  are also true in the more general aperiodic case, but the non-trivial case is the ergodic one. Indeed, in the non-ergodic case we can directly use Lemma \ref{lem: commutator with support A} to build two non-trivial elements in $D([\Phi]_1)$ whose supports are $T$-invariant and disjoint, so  the normal subgroups they generate actually intersect trivially. 
\end{rmq}

\subsection{Density of topological full groups, amenability and finiteness of the topological rank}\label{sec: main thms}

\begin{thm}\label{thm: density of the whole topo full group}Let $T$ be a homeomorphism of the Cantor space, and let $\nu$ be a non-atomic ergodic  probability measure preserved by $T$. Then the topological full group $[T]_c$ is dense in the $\LL^1$ full group $[T]_1$. 
\end{thm}
\begin{proof}
By Theorem \ref{thm: density gp gen by invol in topo ful group}, the group  $\overline{[T]_c}$ contains $\overline{D([\Phi]_1})$, in particular it contains every periodic element of $[\Phi]_1$ by Lemma \ref{lem: periodic in derived group}. By Corollary \ref{cor: generated by T along with periodic} the group $[T]_1$ is generated by periodic elements together with $T$, and since $T\in[T]_c$ we conclude that $\overline{[T]_c}=[T]_1$.
\end{proof}

So given a measure-preserving ergodic homeomorphism $T$ of $(2^\N,\lambda)$,  we have the following chain of Polish groups where each is dense in its successor:
$$[T]_c\leq[T]_1\leq [T]\leq \Aut(2^\N,\nu).$$ 
I don't know whether the above theorem can be generalised to other finitely generated groups acting by measure-preserving homeomorphisms. Recall however that this result is wrong for graphings in general (see Example  \ref{example: full topological not dense in L1}).\\

Recall that a topological group is amenable if all its continuous actions on compact spaces admit invariant probability measures. Note that if a topological group contains a dense subgroup which is amenable as a discrete group, it has to be amenable itself. We view countable groups as discrete groups. Juschenko and Monod have shown that topological full groups of minimal $\Z$-actions on the Cantor space are amenable \cite{zbMATH06203677}. Since the Jewett-Krieger theorem ensures us that every ergodic $\Z$-actions can be modelled by a minimal uniquely ergodic homeomorphism (see e.g. \cite{MR1073173}), the previous theorem has the following interesting application.

\begin{thm}\label{thm:amenable for Z}Let $T\in\Aut(X,\mu)$ be ergodic. Then its $\LL^1$ full group $[T]_1$ is an amenable topological group, and so is the closed subgroup $\overline{D([T]_1)}$.\qed
\end{thm}

It would be interesting to obtain a purely ergodic-theoretic proof of the previous result. The same is true of our next result where we rely on Matui's work on finite generation of topological full groups (the commutator group of the topological full group of a minimal homeomorphism is finitely generated if and only if the homeomorphism is a subshift of finite type, see \cite[Thm. 5.4]{MR2205435}). We should also point out that our proof of (1)$\impl$(2) is a natural adaptation of his.

\begin{thm}Let $T\in\Aut(X,\mu)$ be ergodic. Then the following are equivalent
\begin{enumerate}[(1)]
\item the  $\LL^1$ full group $[T]_1$ is topologically finitely generated;
\item the closed subgroup $\overline{D([T])_1}$ is topologically finitely generated;
\item the transformation $T$ has finite entropy.
\end{enumerate}
\end{thm}
\begin{proof}
(1)$\impl$(3) Suppose that $[T]_1$ is topologically finitely generated, and let $S_1,...,S_k\in[T]_1$ generate a dense subgroup of $[T]_1$. For each $i\in\{1,...,k\}$ and each $n\in\Z$ let $A_{i,n}:=\{x\in X: c_{S_i}(x)=n\}$. Since $S_1,...,S_k\in[T]_1$, for every $i\in\{1,...,k\}$ we have $$\sum_{n\in\Z}\abs n\mu(A_{i,n})<+\infty$$
Lemma \ref{lem: l1 partition has finite entropy} thus yields that every partition $(A_{i,n})_{n\in\N}$ has finite entropy. The partition $\mathcal P:=\bigvee_{i=1}^k(A_{i,n})_{n\in\Z}$  thus has finite entropy. We will show that it is a generating partition. 

Let $\mathcal A$ be the algebra generated by the $T$-translates of $\mathcal P$. Note that every element of the group generated by $S_1,...,S_k$ must have a cocycle partition whose elements belong to $\mathcal A$.

Since $\mathcal P$ has finite entropy, in order to conclude that $T$ has finite entropy we only need to check that $\mathcal A$ is dense in $\MAlg(X,\mu)$. 

For every Borel set $A$ the transformation induced by $T$ on $A$ has support equal to $A$ and can be approximated by elements of the group generated by $S_1,...,S_k$ in the metric $\tilde d_T$ which refines the uniform metric $d_u$. So the supports of these transformations form a sequence $(A_n)$ whose elements belong to $\mathcal Q$ such that $\mu(A_n \bigtriangleup A)\to 0$. We can conclude that up to measure zero, $\mathcal Q$ generates the $\sigma$-algebra of $X$. As explained before, this yields that $T$ has finite entropy. \\

    (3)$\impl$(2) : If $T$ has finite entropy, Krieger's finite generator theorem yields that we can take $T$ to be a minimal subshift \cite{MR0393402}. It then follows from Matui's result \cite[Thm. 5.4]{MR2205435} that $D([T]_c)$ is finitely generated. So by Theorem \ref{thm: density of the whole topo full group} the Polish group $\overline{D([T]_1)}$ is topologically finitely generated.\\ 
    
   (2)$\impl$(1): Suppose  $\overline{D([T]_1)}$ is topologically finitely generated. 
   It follows from Corollary \ref{cor:index take values in Z} that $$[T]_1=\overline{D([T]_1)}\rtimes \la T\ra,$$
    so $[T]_1$ is also topologically finitely generated.
\end{proof}

We now give a concrete example where the topological rank is equal to $2$. Our observation is that the topological generators built by Marks for the whole full group of an irrational rotation actually work for the $\LL^1$ full group. 

\begin{prop}The topological rank of the $\LL^1$ full group of an irrational rotation is equal to $2$.
\end{prop}
\begin{proof}
The topological generators that Marks built in \cite{markstopogen} to show that the full group of the ergodic hyperfinite equivalence relation is topologically $2$-generated can readily be used. 

To be more precise, consider an irrational rotation $T_\alpha$ on $[0,1[$ of irrational angle $\alpha\in[0,1[$. Let $\beta\in(0,\alpha)$ be irrational and consider the involution  $I_{T_\alpha,[0,\beta[}$ as defined in the paragraph preceding Theorem \ref{thm: topo gen by involutions induced by graphing}. 

The exact same proof as Marks' shows that the group topologically generated by $T_\alpha$ and $I_{T_\alpha,[0,\beta[}$ contains all the involutions $I_{T_\alpha, A}$ where $A$ is any Borel subset of $[0,1[$ such that $A\cap T_\alpha(A)=\emptyset$ (indeed the uniform and the $\LL^1$ topology coincide on the set of involutions of the form $I_{T_\alpha, A}$ so Marks' proof readily applies). So the group topologically generated by $T_\alpha$ and $I_{T_\alpha,[0,\beta[}$ contains $\overline{D([T_\alpha]_1)}$ by Theorem \ref{thm: topo gen by involutions induced by graphing}. But  by Corollary \ref{cor:index take values in Z} we have  $$[T_\alpha]_1=\overline{D([T]_1)}\rtimes \la T\ra,$$
so the group topologically generated by $T_\alpha$ and $I_{T_\alpha, A}$ is equal to $[T_\alpha]_1$ which has thus topological rank at most $2$. Since  $[T_\alpha]_1$ is not abelian, we conclude that its topological rank is equal to $2$.
 \end{proof}

\begin{rmq}Matui's work \cite[Prop. 3.2]{MR3103094} can also be used to obtain the above result for irrational rotations of angle $\alpha\in(0,1/6)$. 
\end{rmq}

Once can also adapt an unpublished proof of Marks to show that the topological rank of the $\LL^1$ full group of the odometer is equal to two. More generally, we have a proof that every \emph{rank one} transformation has an $\LL^1$ full groups whose topological rank is equal to two. This applies in particular to compact transformation. The proof will appear in a subsequent paper.

\section{Further remarks and questions}\label{sec: further}

\subsection{Non-amenability}

\begin{thm}\label{thm:fg non amenable}Let $\Gamma$ be a finitely generated non amenable group acting freely in a measure-preserving manner on $(X,\mu)$. Then neither $[\Gamma]_1$ nor $\overline{D([\Gamma]_1)}$ are amenable.
\end{thm}
\begin{proof}
Let $\Gamma$ be a finitely generated non amenable group acting freely in a measure-preserving manner on $(X,\mu)$. It follows from Theorem \ref{thm: L1 is dense} that $[\Gamma]_1$ is dense in $[\mathcal R_\Gamma]$ for the uniform topology. Since $\Gamma$ is infinite and acting freely, the equivalence relation $\mathcal R_\Gamma$ is aperiodic so the group $[\mathcal R_\Gamma]$ is perfect by \cite[Prop. 3.6]{lemai2014}. So $\overline{D([\Gamma]_1)}$ is also dense in $[\mathcal R_\Gamma]$. Furthermore, the group $[\mathcal R_\Gamma]$ is non-amenable for the uniform topology by \cite[Prop. 5.6]{MR2311665}. So the dense subgroups $[\Gamma]_1$ and $\overline{D([\Gamma]_1)}$ are not amenable for the uniform topology, in particular they are not amenable for the thinner $\LL^1$ topology.
\end{proof}

\begin{rmq}
In an upcoming joint work with Carderi and Tanskov, it is shown that the converse of the above result holds: if a countable amenable group $\Gamma$ acts freely in a measure-preserving manner on $(X,\mu)$, then both $[\Gamma]_1$ and $\overline{D([\Gamma]_1)}$ are amenable.
\end{rmq}

\subsection{Absence of a natural closure property}

Full groups as defined by Dye are  stable under cutting and pasting a countable family of elements along a countable partition of $X$.
Here $\LL^1$ full groups are only stable under cutting and pasting a \textit{finite} family of elements, and it seems natural to wonder if they can be more abstractly characterised. Here we show that $\LL^1$ full groups are not ‘‘stable under $\LL^1$ cutting and pasting'', showing that the most natural closure property one could hope for does not hold.

\begin{thm}\label{thm:non stability under l1}There exists an ergodic measure-preserving transformation $T$ and a non-null Borel set $A\subseteq X$ such that $[T_A]_1\not\leq [T]_1$. 
\end{thm}
\begin{rmq}A reformulation of the theorem is that the inclusion ${[T]_1}_A\leq [T_A]_1$ can be strict. In particular, given $S\in [T]_1$, it can happen that $[S]_1\not\leq[T]_1$. 
\end{rmq}
\begin{proof}
We fix a Borel partition $(X_n)_{n\geq 1}$ of  $X$ such that for every $n\in\N$, $\mu(X_n)=2^{-n}$. We then cut each $X_n$ into $4^n$ pieces $(B_{n,m})_{m=0}^{4^n-1}$ of equal measure (so $\mu(B_{n,m})=8^{-n}$). 

Let $T_0$ be a periodic transformation such that 
\begin{itemize}
\item for all $n\in\N$ and $m\in\{0,...,4^n-2\}$, one has $T_0(B_{n,m})=B_{n,m+1}$ and
\item for all $n\in\N$ and $x\in B_{n,0}$ one has $T_0^{4^n+1}(x)=x$.
\end{itemize}

Then the set $B_0:=\bigsqcup_{n\in\N} B_{n,0}$ is a fundamental domain for $T_0$. Let $U$ be a measure-preserving transformation such that $\supp U=B_0$ and $U_{\restriction B_0}$ is ergodic, and consider $T:=UT_0$. By 
Proposition \ref{prop: building ergodic and aperiodic from periodic}, $T$ is ergodic. Moreover for all $n\in\N$ and $m\in\{0,...,4^n-1\}$ we  have $T(B_{n,m})=B_{n,m+1}$.
Now for all $n\in\N$ we let
\begin{align*}
A_n&:=\bigsqcup_{m=0}^{2^n-1} B_{n,2^nm}\end{align*}
and then $\displaystyle A:=\bigsqcup_{n\in\N} A_n$.

Consider the transformation $T_A$, and note that by construction for all $n\in\N$, all $m\in\{0,2^n-2\}$ and all $x\in B_{n,2^nm}$ we have $T_A(x)=T^{2^n}(x)$. 
Now consider the involution $U_n$ defined by 
$$U_n(x)=\left\{\begin{array}{cl}
T^{4^n/2}(x)&\text{if }x\in \bigsqcup_{m=0}^{2^n/2-1} B_{n,2^nm} \\
T^{-4^n/2}(x) & \text{if }x\in\bigsqcup_{m=2^n/2}^{2^n-1} B_{n,2^nm}\\
x & \text{else.}
\end{array}\right.$$
\begin{figure}[h]
\centering
\includegraphics{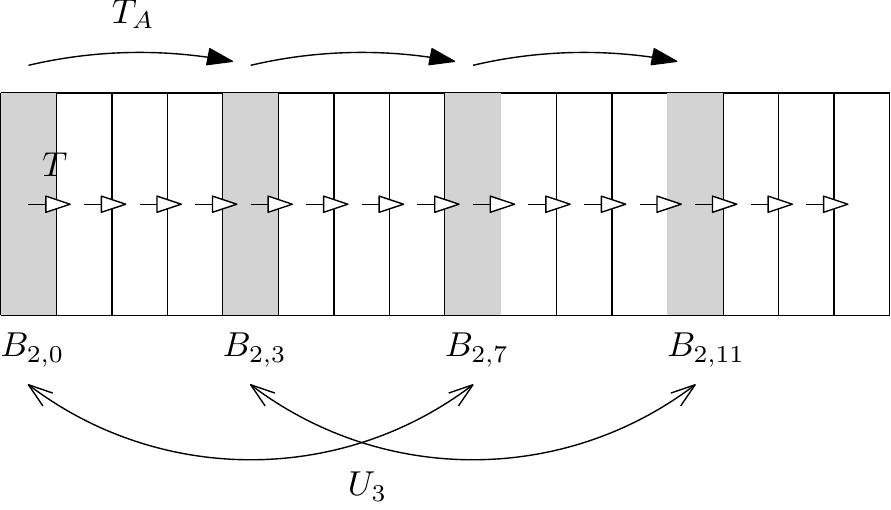}
\caption{The construction of $U_3$. The set $A_3$ is the gray part of the figure.}
\end{figure}

Note that for all $x\in A_n$ we have $d_T(x,U_n(x))=4^n/2$. 
We can then compute \begin{align*}d_T(U_n,\mathrm{id}_X)&=\sum_{m=0}^{2^n-1} \int_{B_{n,2^nm}}d_T(x,U_n(x))\\
&=\frac{2^n}{8^n}\times\frac{4^n}2\\
&=\frac 12
\end{align*}
Also for all $x\in A_n$ we have $d_{T_A}(x,U_n(x))=2^{n}/2$ so the same computation as before yields
$$d_{T_A}(U_n,\mathrm{id}_X)=\frac{2^n}{8^n}\times \frac{2^n}2=\frac 1{2^{n+1}}$$
Let $U$ be the involution obtained by gluing together all the $U_n$'s, namely for all $x\in X$
$$U(x)=\left\{\begin{array}{cl}U_n(x) & \text{if }x\in A_n \text{ for some }n\in\N,  \\x & \text{else.}\end{array}\right.$$
We clearly have
\begin{align*} d_T(U,\mathrm{id}_X)&=\sum_{n\in\N}d_T(U_n,\mathrm{id}_X)=\sum_{n\in\N}\frac12=+\infty\\
\text{and }d_{T_A}(U,\mathrm{id}_X)&=\sum_{n\in\N}d_{T_A}(U_n,\mathrm{id}_X)=\sum_{n\in\N}\frac1{2^{n+1}}=\frac 12<+\infty.
\end{align*}
We conclude that $U$ witnesses the fact that $[T_A]_1\not\leq [T]_1$.
\end{proof}

Recall that two actions of finitely generated groups $\Gamma$ and $\Lambda$ are $\LL^1$ orbit equivalent if up to a measure-preserving transformation, they share the same orbits and for every $\lambda\in \Lambda$ and $\gamma\in\Gamma$, the maps
$$x\mapsto d_\Gamma(x,\lambda(x))\text{ and }x\mapsto d_\Lambda(x,\gamma(x))$$ are integrable, where $d_\Gamma$ and $d_\Lambda$ are respective Schreier graph distances (see e.g. \cite{Austin:2016rm} for more on this notion). Here we note that although sharing the same $\LL^1$ full group implies $\LL^1$ orbit equivalence, the converse does not hold a priori. 

\begin{cor}\label{cor: L1 OE but not same L1 full group}There exists two measure-preserving actions of the free group on two generators $\mathbb F_2$ such that the identity map is an $\LL^1$-orbit equivalence between them but they do not share the same $\LL^1$ full group. 
\end{cor}
\begin{proof}Write $\mathbb F_2=\la a,b\ra$, fix $T\in\Aut(X,\mu)$ and $A\subseteq X$ as in the proof of the previous theorem. Define a measure-preserving $\mathbb F_2$-action $\alpha_1$ by $\alpha_1(a)=T$ and $\alpha_2(b)=T$. Then define a second action $\alpha_2$ by $\alpha_2(a)=T_A$  and $\alpha_2(b)=TT_A\inv$. Since $T_A\in [T]_1$ these two actions are $\LL^1$ orbit equivalent via the identity map. However a map $U$ as in the proof of the previous theorem belongs to the $\LL^1$ full group of $\alpha_2$, but not to the $\LL^1$ full group of $\alpha_2$. 
\end{proof}

\begin{rmq} Obviously $\LL^\infty$ orbit equivalence implies conjugacy of the $\LL^1$ full groups. It would be interesting to understand wether the converse holds. 
\end{rmq}

\subsection{$\LL^p$ full groups}

Let $p\in [1,+\infty]$. One can then define $\LL^p$ full group $[\Phi]_p$ of a graphing $\Phi$ on a standard probability space $(X,\mu)$ by letting 
$$[\Phi]_p=\left\{T\in[\mathcal R_\Phi]:  \text{the map }x\mapsto d_\Phi(T(x),x)\text{ belongs to }\LL^p(X,\mu)\right\}.$$
The triangle inequality and Minkowski's inequality yield that $[\Phi]_p$ is indeed a group, and the natural distance $\tilde d^p_\Phi$ given by 
$$\tilde d^p_\Phi(S,T)=\norm{x\mapsto d_\Phi(S(x),T(x))}_p$$
is right-invariant, complete and refines the uniform topology. Moreover Hölder's inequality implies that $[\Phi]_q\leq [\Phi]_p$ whenever $p<q$. Note that the $\LL^\infty$ full group is discrete since $d_\Phi$ takes values in $\Z$. It is also uncountable hence we view it as a degenerate case. For $p<+\infty$ however, one can show exactly like for $p=1$ that $[\Phi]_p$ is separable, so that it is also a Polish group. 


The main difference with $\LL^1$ full groups is that for $p>1$ the $\LL^p$ full groups are not stable under taking induced transformations. Indeed, given a set $A$ such that $0<\mu(A)<1$ and a map $f:A\to \N$ of integral $1$, one can  use Lemma \ref{lem: induced from periodic} to build a measure-preserving ergodic transformation $T$ whose return time to $A$ equates $f$. So if $f$ was chosen not to be in $\LL^p$ the transformation induced by $T$ on $A$ won't belong to the $\LL^p$ full group of $T$. \\

If $T$ is an ergodic measure-preserving transformation, the $\LL^p$ full group of $T$ is a complete invariant of flip conjugacy since it also has many involutions and is a subgroup of the $\LL^1$ full group.  The index map on $[T]_1$ restricts to its $\LL^p$ full group and yields a surjective homomorphism $[T]_p\to \Z$. The kernel of this homomorphism clearly contains $\overline{D([T]_p)}$ but it is unclear wether the converse holds.

Since $\overline{D([T]_1)}=\ker I$, a positive answer to the following general question would also imply that $\overline{D([T]_p)}=\ker I$.

\begin{qu}Let $\Phi$ be an ergodic graphing. Is it true that for all $p\in]1,+\infty[$ we have the equality
$$\overline{D([\Phi]_1)}\cap [\Phi]_p=\overline{D([\Phi]_p)},$$
where the closure in the left term is in the $\LL^1$ topology while the closure in the right term is in the $\LL^p$ topology?
\end{qu}

I also do not know wether $\overline{D([\Phi]_p)}$ is always generated by involutions, which if true would yield via the same proof as Thm. \ref{thm: density gp gen by invol in topo ful group} that $\overline{D([\Phi]_q)}$ is dense in $\overline{D([\Phi]_p)}$ for $1\leq p<q\leq+\infty$.

\subsection{Questions}

We close this paper by recalling some questions scattered through the text and asking a few more. Answering the first is would provide a much better understanding of $\LL^p$ full groups.

\begin{qu}Let $1\leq p<q\leq +\infty$. Is it true that $[\Phi]_q$ is dense in $[\Phi]_p$? 
\end{qu}

Note that a positive answer for $q=+\infty$ yields a positive answer for every $1\leq p<q\leq+\infty$. When $\Phi$ is a continuous graphing, we obviously have $[\Phi]_c\leq [\Phi]_\infty$. We thus also ask: 

\begin{qu}Let $\Phi$ be a finite continuous graphing and $1\leq p<+\infty$. Is it true that $[\Phi]_c$ is dense in $[\Phi]_p$?
\end{qu}

Note that the answer is negative if we remove the condition that $\Phi$ is finite (cf. Example \ref{example: full topological not dense in L1}). We have seen that the $\LL^1$ full group of an ergodic measure-preserving transformation $T$ is topologically finitely generated if and only if $T$ has finite entropy, and gave examples where the topological rank of $[T]_1$ is equal to $2$. All these examples have entropy zero. One could hope that there is a formula linking the topological rank to entropy as is true for cost and topological rank of full groups \cite{gentopergo}, but for now we don't even know the answer to the following very basic and pessimistic question.

\begin{qu}Let $T$ be an ergodic measure-preserving transformation of finite entropy. Is the topological rank of $[T]_1$ equal to $2$?
\end{qu}

It would be interesting to try to generalise our structural results on $\LL^1$ full groups of ergodic $\Z$-actions to $\Z^n$-actions. For instance, given a measure-preserving ergodic $\Z^n$-action, one can still define an index map $I_n:[\Z^n]_1\to \R^n$ and ask for analogues of Corollary \ref{cor:index take values in Z}.

\begin{qu}\label{qu: index for higher dimension}Does the index map $I_n$ take values into $\Z^n$ when the canonical generators of $\Z^n$ act ergodically? Is the kernel of $I_n$ equal to the closure of the derived group of $[\Z^n]$?
\end{qu}

Our positive answer to this question for $n=1$ relied crucially on the decomposition of elements of $[T]_1$ as products of almost positive, almost negative and periodic elements of disjoint supports, but it is not clear what an analogue of this in higher dimension would be so that another approach is probably needed. 

We should mention that in the case of a single ergodic measure-preserving transformation $T$, one can use the ideas from section \ref{sec:generated by induced and periodic} to show that the open unit ball in $[T]_1$ only consists of periodic elements. Since those belong to $\overline{D([T]_1)}$ (cf. Lemma \ref{lem: periodic in derived group}), this yields another way to prove that $\overline{D([T]_1)}$ is open. We thus ask:

\begin{qu}
Let a finitely generated group $\Gamma$ act freely on $(X,\mu)$ with all its generators acting ergodically. When does the open unit ball in $[\Gamma]_1$ only contain periodic elements? 
\end{qu}

The fact that for an ergodic measure-preserving transformation $T$ the open unit ball of the $\LL^1$ full group only contains periodic elements can be used to show that $[T]_1$ has Rosendal's local property (OB) and is (OB) generated, so that by Rosendal's work it has a well defined quasi-isometry type as a topological group (see \cite{rosendalcoarse2016}). This will be the subject of a later work.\\ 

Finally, full groups of ergodic measure-preserving equivalence relations can be seen as special cases of $\LL^1$ full groups (see Example \ref{ex: full group of R}). They enjoy the automatic continuity property as was shown by Kittrell and Tsankov \cite[Thm. 1.5]{MR2599891} and hence admit a unique Polish group topology. 

\begin{qu}
Let $\Phi$ be an ergodic graphing. Does $[\Phi]_1$ satisfy the automatic continuity property? Is its Polish group topology unique? What about $\overline{D([\Phi]_1)}$ ?
\end{qu}


\bibliographystyle{alpha}
\bibliography{/Users/francoislemaitre/Dropbox/Maths/biblio}

\end{document}